\documentclass[11pt,reqno]{amsart}
\usepackage[utf8]{inputenc}
\usepackage{fullpage}

\usepackage{graphicx}
\usepackage{hyperref}
\hypersetup{colorlinks,allcolors=black}
\usepackage{cleveref}
\usepackage{amsmath,amsopn,amssymb,amsfonts,stmaryrd}
\usepackage{verbatim}
\usepackage{amsthm}
\usepackage{tikz-cd}
\usepackage{tikz}
\usepackage{dsfont}
\usetikzlibrary{arrows}
\usepackage{mathtools}
\usepackage{color}
\usepackage{enumitem}
\usepackage[framemethod=TikZ]{mdframed}
\usepackage{bbm}
\usepackage{mathrsfs}
\usepackage{booktabs}
\usepackage{caption}
\usepackage{bm}
\usepackage{tensor}
\usepackage[cal=boondoxo]{mathalfa}
\usepackage[all,cmtip]{xy}

\makeatletter
\@namedef{subjclassname@2020}{%
	\textup{2020} Mathematics Subject Classification}
\makeatother

\newcommand{\CC}{\mathcal{C}}

\newcommand{\CP}{\mathcal{P}}

\renewcommand{\Im}{{\rm Im}}

\newcommand{\tr}{\mathrm{tr}}

\newcommand{\GL}{\mathrm{GL}}

\newcommand{\C}{\mathbb C}

\newcommand{\D}[1]{\ensuremath{\mathrm{D}#1}}

\newcommand\restr[2]{\ensuremath{\left.#1\right|_{#2}}}

\theoremstyle{plain}
\newtheorem{thm}{Theorem}[section]
\newtheorem{cor}[thm]{Corollary}
\newtheorem{lem}[thm]{Lemma}
\newtheorem{prop}[thm]{Proposition}

\newtheorem{ques}[thm]{Question}

\newtheorem{rem}[thm]{Remark}
\newtheorem{rems}[thm]{Remarks}

\theoremstyle{definition}

\newtheorem{defn}[thm]{Definition}

\newtheorem{claim}{Claim}
\numberwithin{equation}{section}

\def\d{\delta}

\def\s{\sigma}

\def\d{\delta}

\def\s{\sigma}

\def\CC{{\mathbb C}}
\def\RR{{\mathbb R}}

\def\ZZ{{\mathbb Z}}
\def\d{{\mathrm{d}}}

\def\M{{\mathcal{M}}}
\def\MAC{\mathcal{M}_{\dot{H}^1}}
\def\GH{\mathcal{G}^{\dot{H}^1}}
\def\GL{\mathcal{G}^{L^2}}

\def\id{\mathrm{id}}

\def\Sd{{\mathbb{S}^3}}

\def\SS{{\mathbb{S}}}
\def\S{{\mathbb{S}^{2n+1}}}
\def\Sil{{\mathbb{S}_{l^2}^{\infty}}}
\def\SiL{{\mathbb{S}_{L^2}^{\infty}}}
\def\SiLm{\mathbb{S}_{L^2}^{\infty, m}}

\def\i{{\mathrm{i}}}

\def\Re{{\mathrm{Re}}}
\def\Im{{\mathrm{Im}}}

\def\D{\mathrm{D}}
\def\II{\mathrm{\RNum{2}}}

\def\Diff{\mathrm{Diff}}
\def\U{\mathcal{U}}
\def\tu{\Tilde{u}}
\def\tr{\Tilde{\rho}}

\def\Lt{L^2\left(\SS^1, \CC \right)}
\def\CP{\CC\mathrm{P}^{\infty}}
\def\vardalp{\varPhi^{*}\d\alpha}
\def\JHo{J^{\dot{H}^1}}
\def\XiHo{\xi_{\dot{H^1}}}
\newcommand{\gc}{\gamma}
\newcommand{\dgc}{\dot{\gc}}

\setlist[itemize]{noitemsep, topsep=0pt}

\makeatletter
\newcommand{\vast}{\bBigg@{2}}
\newcommand{\Vast}{\bBigg@{5}}
\makeatother

\newcommand{\RNum}[1]{\uppercase\expandafter{\romannumeral #1\relax}}
\allowdisplaybreaks

\title{On Mañé's critical value for the two-component Hunter-Saxton system and an infinite-dimensional magnetic Hopf--Rinow theorem} 

\author{L. Maier}
\address{Faculty of Mathematics and Computer Science,
	University of Heidelberg,
	Im Neuenheimer Feld 205,
	D-69120 Heidelberg, Germany}
\email{lmaier@mathi.uni-heidelberg.de}

\begin{document}
\renewcommand{\abstractname}{Abstract}
\begin{abstract}
In this paper, we introduce a nonlinear system of partial differential equations, the magnetic two-component Hunter--Saxton system (M2HS). This system is formulated as a magnetic geodesic equation on an infinite-dimensional Lie group equipped with a right-invariant metric, namely the $\dot{H}^1$-metric, which is closely related to the infinite-dimensional Fisher--Rao metric. The magnetic field is given by the exterior derivative of an infinite-dimensional contact-type form.\\[0.3em]
We define Mañé's critical value for exact magnetic systems on Hilbert manifolds in full generality and compute it explicitly for the (M2HS). Moreover, we establish an infinite-dimensional Hopf--Rinow theorem for this magnetic system, where Mañé's critical value serves as the threshold beyond which the Hopf--Rinow property fails.\\[0.3em]
This geometric framework enables a detailed analysis of the blow-up behavior of solutions to the (M2HS). Using this insight, we extend solutions beyond blow-up by introducing and proving the existence of global conservative weak solutions. This extension is achieved by extending the Madelung transform from an isometry to a magnetomorphism, thereby embedding the magnetic system into a magnetic system on an infinite-dimensional sphere equipped with the exterior derivative of the standard contact form as the magnetic field.\\[0.3em]
Crucially, this setting can always be reduced, via a dynamical reduction theorem, to a totally magnetic three-sphere, providing deeper insight into the underlying dynamics.
\end{abstract}
\maketitle
 \section{Introduction}
In his seminal work \cite{ar66}, V. Arnold discovered that the Euler equation in hydrodynamics, which describes the motion of an incompressible and inviscid fluid occupying a fixed domain (with or without boundary), is precisely the geodesic equation on the group of volume-preserving diffeomorphisms of the domain equipped with a right-invariant Riemannian metric, the $L^2$ metric. Notably, the analytical study of the differential geometry of the incompressible Euler equations began by Ebin--Marsden in \cite{EM70}.

From the perspective of differential geometry, it is natural to ask: If we change the infinite-dimensional Lie group and/or the Riemannian metric, which partial differential equations (PDEs) can also be described as geodesic equations? This has been investigated intensively in the past. For a general overview, we refer to the monograph by Arnold–-Khesin \cite{AK98} and the survey articles \cite{Khesin-Mis-Mod-inf-Newton, Vi08}.

Before we proceed, we would like to highlight an important example relevant to this work: the Hunter--Saxton equation, which is the geodesic equation on the Virasoro group equipped with the $\dot{H}^1$-metric. This was discovered by Khesin--Misiolek in \cite{km02} and later used in \cite{l07.1, l07.2} to prove, for instance, the existence of global weak solutions for this PDE.

We now move on to discuss a natural generalization of the geodesic equation from the perspective of Hamiltonian dynamics: Newton's equation. Nowadays, several PDEs from mathematical physics possess a formulation as an infinite-dimensional Newton equation, including the geodesic equation. Notable examples include the barotropic inviscid fluid, fully compressible fluid, magnetohydrodynamics, inviscid Burgers’ equation, the Hamilton–Jacobi equation, the Klein–Gordon equation, the KdV equation, and the Camassa–Holm equation. For a general overview of the landscape of this topic, we refer to \cite{KhesinMisilokeModinMadelung-proc, Khesin-Mis-Mod-inf-Newton}.

In physical terms, Newton's equation describes the motion of a particle under the influence of a potential force. Thus, from a physical point of view, the next natural step is to investigate the motion of a charged particle in a magnetic field. In mathematics, this falls under the framework of Hamiltonian dynamics, specifically magnetic systems. Following the pioneering work of V. Arnold \cite{ar61} in the 1960s, finite-dimensional magnetic systems have garnered significant attention in recent years; see, for example, \cite{AbbMacMazzPat17, AssBenLust16, Co06, CIPP, Gin, Man}. It is noteworthy that the trajectories of magnetic systems are called \emph{magnetic geodesics}, which are solutions to the corresponding equation of motion known as the \emph{magnetic geodesic equation}; see \Cref{d: Magnetic system}.

This naturally leads to the following question, arising from the preceding discussion of various PDEs as geodesic or Newton equations:
\begin{quote}
Are there infinite-dimensional magnetic systems such that the associated magnetic geodesic equation is a PDE?
\end{quote}

We provide a positive answer to this question in this article by introducing the \emph{magnetic two-component Hunter-Saxton system} \eqref{(M2HS)}. Before stating the main contributions of this paper, let us emphasize the following: A key point of interest is to explore the similarities and differences between standard and magnetic geodesics, in this paper by means of a Hopf--Rinow theorem. According to this pivotal result, every two points on a closed connected finite-dimensional manifold can be connected by a standard geodesic \cite{Pet}. However, this finding does not generally hold for Hilbert manifolds \cite{HopfrinowfalseAktkin,HopfrinowfalseEkeland}, but it is surprisingly true for half-Lie groups due to the recent work of Bauer-Harms-Michor \cite{Bauer_2025}. 

Notably, a standard fact in differential geometry ensures that the existence of a connecting geodesic is independent of its speed or energy. Therefore, it is natural to ask whether there exists a magnetic geodesic connecting two given points with a prescribed kinetic energy. An important role in addressing this question is played by the Mañé critical value, which, until now, has been defined only for finite-dimensional magnetic systems \cite{CIPP,Man}. \\
\noindent 
\\
\textbf{The main contributions of this article} are as follows:

\begin{itemize}
		\item \textbf{Mañé's critical value for Hilbert manifolds:} 
	We introduce Mañé's critical value for for exact magnetic systems on Hilbert manifolds (\Cref{Def: Mañé's critical value or Hilber manifolds}). 
	
	\item \textbf{The dynamical reduction theorem:} We compute Mañé's critical value in \Cref{L: Mane crit value for SIL} for a natural infinite-dimensional magnetic system: the unit sphere in the space of square-integrable complex-valued functions, \(\SiL\), equipped with the round metric and an infinite-dimensional contact form defined in \Cref{def:contact_form_infinite_dimension}. Using a dynamical reduction theorem (\Cref{cor: dynamical reduction of SiL to Sd}), we reduce the dynamics to a totally magnetic three-sphere, as defined in \Cref{Definition totally magnetic submanifold}, and establish an infinite-dimensional Hopf--Rinow theorem for the magnetic system. The Mañé critical value marks the threshold at which the Hopf--Rinow theorem no longer applies (\Cref{t:mane for SiL}).
	
	\item \textbf{The Madelung transform as a magnetomorphism:} 
	We extend the Madelung transform from being an isometry of an infinite-dimensional Lie group \(G\), equipped with a right-invariant metric, to an open subset $\U$ of the unit sphere in the space of complex-valued square-integrable functions. This transform is demonstrated to be a \emph{magnetomorphism} in the sense of \Cref{d: magnetomorphism}, as shown in \Cref{t: magnetic Hunter Saxton system}. As a special case, we recover the corresponding results in \cite{l07.1, Lennels13}. Our choice of a magnetic field is derived naturally through the lens of infinite-dimensional contact geometry, transforming \(G\) into an infinite-dimensional Boothby-Wang bundle over the cotangent bundle of probability densities on \(\mathbb{S}^1\), equipped with its standard symplectic form (\Cref{C: TDens as boothby wang bundel}).
	
	\item \textbf{A PDE as a magnetic geodesic equation:} 
	We provide a positive answer to the question posed earlier by introducing the \emph{magnetic two-component Hunter-Saxton system} \eqref{(M2HS)} and proving in \Cref{t: magnetic Hunter Saxton system} that solutions of \eqref{(M2HS)} are in one-to-one correspondence with the magnetic geodesics of an infinite-dimensional magnetic system. To the best of the author's knowledge, this is the first instance in the literature where a PDE is shown to be precisely a magnetic geodesic equation on an infinite-dimensional Lie group.
	
	\item \textbf{Geometry of blow-ups and global conservative weak solutions:} 
	Using the framework of infinite-dimensional Hamiltonian dynamics, we explain the occurrence of blow-ups in the \eqref{(M2HS)}(\Cref{t: geometric interpret of blow ups}), with visualizations in \Cref{f: geometric blow up criteria}. Leveraging the Madelung transform as a magnetomorphism (\Cref{t: Madelungtransform as Magnetomorphims}) and the aforementioned dynamical reduction theorem (\Cref{cor: dynamical reduction of SiL to Sd}), we can deduce from the initial data of a solution of \eqref{(M2HS)} whether a blow-up will occur. More intriguingly, this geometric framework enables us to establish a notion of global conservative weak solutions under mild regularity assumptions (\Cref{d: weak solutions}) and provides a natural explanation for their global existence (\Cref{t: global weak solutions }). Analytical details are provided in \cite{Maier26Globalweak}, and as a special case, we recover results from \cite{l07.1, l07.2, wu11}.
	
	\item \textbf{Mañé's value for the (M2HS):} 
	We compute Mañé's critical value for the natural completion \(\MAC\) of the magnetic system on the Lie group \(G\), which we have developed in \Cref{t: global weak solutions }. As an illustration, we show that the Mañé critical value serves as the threshold where the Hopf--Rinow theorem ceases to hold for this magnetic system (\Cref{t: Manes critical value for (M2HS)}), analogous to \Cref{t:mane for SiL}. Using \Cref{t: magnetic Hunter Saxton system}, we provide a Hopf--Rinow theorem for solutions of \eqref{(M2HS)}.
\end{itemize}
\noindent
\\
A visual representation of the current state of these contributions is provided in the following diagram:

\begin{equation}\label{f: state of th art}
	\begin{tikzpicture}
		\node (A) at (0, 2) {$(\MAC, \GH,\Phi^* \d\alpha)$};
		\node (B) at (7, 2) {$(\U, \GL, \d\alpha)\subseteq (\SiL, \GL, \d\alpha) $};
		\node (D) at (7, -2) {$(S^3, g, \d\alpha)$};
		\node (E) at (0, -2) {$T_{\mathrm{id}} \MAC$};
		
		\draw[->] (A) -- node[above] {$\cong$} node[below]{$\mathrm{Madelung}$} (B);
		\draw[->, bend left=45] (A) to (E);
		\draw[->, bend left=45] (E) to (A);
		\draw[right hook->] (D) -- (B);
		\draw[dashed, ->] (D)--(E);
		
	\end{tikzpicture}
\end{equation}

The two arrows between the magnetic system \((\MAC, \GH, \Phi^* \d\alpha)\) and \(T_{\mathrm{id}} \MAC\) on the left side of \eqref{f: state of th art}represent the one-to-one correspondence between magnetic geodesics (respectively, weak magnetic geodesics) in \((\MAC, \GH, \Phi^* \d\alpha)\) and solutions of \eqref{(M2HS)}, as established in \Cref{t: magnetic Hunter Saxton system} and \Cref{t: global weak solutions }. 
Starting with a global conservative weak solution \((u, \rho)\) of \eqref{(M2HS)}, which lives in \(T_{\mathrm{id}} \MAC\), we obtain a corresponding magnetic geodesic \((\varphi, \tau)\) in \((\MAC, \GH, \Phi^* \d\alpha)\) via the aforementioned duality. Using the Madelung transform \(\varPhi\) (depicted by the arrow in the diagram), which is a magnetomorphism according to \Cref{t: Madelungtransform as Magnetomorphims}, \(\varPhi(\varphi, \tau)\) becomes a magnetic geodesic in \((\U, \GL, \d\alpha)\). 
Since \(\U \subset \SiL\) is an open subset, it follows that \(\varPhi(\varphi, \tau)\) is also a magnetic geodesic in \((\SiL, \GL, \d\alpha)\). By the dynamical reduction theorem (\Cref{cor: dynamical reduction of SiL to Sd}), this magnetic geodesic remains confined to a totally magnetic three-sphere, depicted by the embedding arrow on the right side of the diagram in \eqref{f: state of th art}.

By traversing all the arrows in this order, the diagram provides a means to translate problems concerning the highly nonlinear PDE \eqref{(M2HS)} into questions about the magnetic geodesic flow on a three-sphere. This magnetic geodesic flow is fully understood via the Hopf--Rinow theorem, as established in previous work by Albers, Benedetti, and the author \cite{ABM23}. 

In summary, we utilize the structure of the magnetic geodesic flow on the three-sphere to address the challenges presented by \eqref{(M2HS)}, as indicated by the dashed arrow in the lower row of \eqref{f: state of th art}.

\noindent
\\
\textbf{Acknowledgments.}
The author is indebted to P. Albers and B. Khesin for valuable discussions throughout the project and for their constant interest in his work. The author thanks P. Albers and G. Benedetti for their collaboration in \cite{ABM23}, which was a source of inspiration for this work.  Appreciation is extended to Ana Chavez Caliz for providing all the figures in this document. The author also acknowledges L. Assele and G. Benedetti for insightful discussions on Mañé’s critical value. Thanks are due to M. Bauer, A. Chavez Caliz, L. Deschamps, K. Modin, S. Preston, and C. Vizman for their helpful comments on an earlier version of this work.
The author further acknowledges funding by the Deutsche Forschungsgemeinschaft (DFG, German Research Foundation) – 281869850 (RTG 2229), 390900948 (EXC-2181/1) and 281071066 (TRR 191). The author gratefully acknowledges support from the Simons Center for Geometry and Physics, Stony Brook
University at which some of the research for this paper was performed during the program
\textit{Mathematical Billiards: at the Crossroads of Dynamics, Geometry, Analysis, and Mathematical Physics.} The author thanks also V. Assenza and V. Ramos for their warm hospitality at IMPA in Rio de Janeiro in February 2024 where some of the research for this paper was performed.  
\section{ Magnetic systems and Mañé critical value}\label{S: 2}
\subsection{Intermezzo: Magnetic systems}In the 1960s, the motion of a charged particle in a magnetic field was placed within the framework of modern dynamical systems by V.\ Arnold in his pioneering work \cite{ar61}. The motion has the following mathematical description:

\begin{defn}\label{d: Magnetic system}
	Let \((M,g)\) be a connected Riemannian Hilbert manifold, and let \(\sigma \in \Omega^2(M)\) be a closed two-form. The form \(\sigma\) is called a \emph{magnetic field}, and the triple \((M,g,\sigma)\) is called a \emph{magnetic system}. This defines the skew-symmetric bundle endomorphism \(Y \colon TM \to TM\), called the \emph{Lorentz force}, by:
	\begin{equation}\label{e:Lorentz}
		g_q\left(Y_qu, v\right) = \sigma_q(u, v), \qquad \forall\, q \in M,\ \forall\, u, v \in T_qM.
	\end{equation}
	A smooth curve \(\gamma \colon \mathbb{R} \to M\) is called a \emph{magnetic geodesic} of \((M,g,\sigma)\) if it satisfies:
	\begin{equation}\label{e:mg}
		\nabla_{\dot{\gamma}} \dot{\gamma} = Y_{\gamma} \dot{\gamma},
	\end{equation}
	where \(\nabla\) denotes the Levi-Civita connection of the metric \(g\).
\end{defn}

\begin{rems}
	\item \textbf{Special case: Geodesics.} From \eqref{e:mg}, it is evident that a magnetic geodesic with \(\sigma = 0\) reduces to a standard geodesic of the metric \(g\). Therefore, \eqref{e:mg} can be interpreted as a linear perturbation of the geodesic equation. A key point of interest is to explore the similarities and differences between standard and magnetic geodesics.
\end{rems}

Since $Y$ is skew-symmetric, magnetic geodesics have constant kinetic energy $E(\gamma,\dot\gamma):=\tfrac12g_\gamma(\dot\gamma,\dot\gamma)$, and hence constant speed $|\dot\gamma|:=\sqrt{g_\gamma(\dot\gamma,\dot\gamma)}$, just like standard geodesics.	
Energy conservation is a footprint of the Hamiltonian nature of the system. Indeed, let us define the \emph{magnetic geodesic flow} on the tangent bundle by
	\[
	\varPhi_{g,\sigma}^t\colon TM\to TM,\quad (q,v)\mapsto \left( \gc_{q,v}(t),\dgc_{q,v}(t)\right),\quad \forall t\in\RR,
	\] where $\gc_{q,v}(t)$ is the unique magnetic geodesic with initial values $(q,v)\in TM$. By \cite{Gin}, $\Phi^t_{g,\sigma}$ is the Hamiltonian flow induced by the kinetic energy $E\colon TM\to\RR$ and the twisted symplectic form 
	\begin{equation}\label{e: twisted symplectic form}
	    	\omega_\sigma:=\d\lambda-\pi^*_{TM}\sigma,
	\end{equation}
	where $\lambda$ is the metric pullback of the canonical Liouville $1$-form from $T^*M$ to $TM$ and $\pi_{TM}\colon TM\to M$ is the projection.	However, differently from the case of standard geodesics, magnetic geodesics of different speeds are not just reparametrization of unit speed magnetic geodesics. For instance, if $M=\SS^2$, $g$ has constant curvature $1$ and $\sigma$ is the corresponding area form, then magnetic geodesics of kinetic energy $k$ are geodesic circles of radius $\arctan(\sqrt
	{2k})$ traversed counterclockwise, see \cite{BKmag}. 
	\begin{rem}
	\item \textbf{A different perspective.} For every \(s > 0\), a curve \(\gamma\) is a magnetic geodesic of \((M, g, \sigma)\) with speed \(\frac{1}{s}\) if and only if the unit-speed reparametrization of \(\gamma\) is a magnetic geodesic of \((M, g, s\sigma)\). This implies that studying magnetic geodesics of \((M, g, \sigma)\) with varying speeds is equivalent to studying unit-speed magnetic geodesics of \((M, g, s\sigma)\) for different values of \(s > 0\). This approach has the added benefit that:
\begin{itemize}
	\item For \(s = 0\), we recover standard geodesics.
	\item For \(s < 0\), we recover magnetic geodesics of \((M, g, \sigma)\) with reversed orientation.
\end{itemize}
In this context, the parameter \(s\) is referred to as the \emph{strength} of the magnetic geodesic.
	\end{rem}
\subsection{Mañé's Critical Value of infinite-dimensional magnetic systems}\label{Manes critical for Hilbert manifolds}
Let $(M,g)$ be a Hilbert manifold endowed with a Riemannian metric $g$, which may be either weak or strong. Let $\alpha$ be a one-form on $M$ and consider the exact magnetic system $(M,g,\d\alpha)$. In this case, the magnetic geodesic flow $\Phi_{g,\sigma}$ is the Euler--Lagrange flow $\Phi_L$ of the magnetic Lagrangian
\begin{equation}\label{magnetic Langrangian}
    L\colon TM\to\RR,\quad L(q,v):=\tfrac12|v|^2-\alpha_q(v),
\end{equation}
see \cite{Gin}. This means that $\gamma\colon[0,T]\to M$ is a magnetic geodesic if and only if $\gamma$ is a critical point of the action functional $S_L$ among all curves $\delta\in H^1\left([0,T],M\right)$ such that $\delta(0)=\gamma(0)$ and $\delta(T)=\gamma(T)$. Here the action functional is defined as
\[
S_L:H^1\left([0,T],M\right)\longrightarrow \RR: \gamma\mapsto S_L(\gamma):=\int_0^TL(\gamma(t),\dot\gamma(t))\d t.
\]
This variational principle prescribes the length $T$ of the time interval and leaves free the energy of $\gamma$. On the other hand, given $k\in\mathbb{R}$, $\gamma$ is a magnetic geodesic with energy $k$ if and only if $\gamma$ is a critical point of $S_{L+k}$ among all curves $\delta\colon [0,T']\to M$ of Sobolev class $H^1$ with $\delta(0)=\gamma(0)$ and $\delta(T')=\gamma(T)$ for an arbitrary $T'>0$. Now we can define Mañé's critical value of the Lagrangian $L$: 
\begin{defn}[Mañé's critical value]\label{Def: Mañé's critical value or Hilber manifolds} The Mañé's critical value $c(L)\in \mathbb{R}$ of $L$ is the smallest energy value such that the free-period action functional with energy $k$ is bounded from below (by zero) on the space of loops of Sobolev class $H^1$: \begin{equation}\label{d
} c(L)=\inf\left\{k\in\mathbb{R}\mid S_{L+k}(\gamma)\geq 0,\ \forall T>0,\ \forall\gamma\in H^1\left([0,T],M\right),\ \gamma(0)=\gamma(T)\right\}. \end{equation} \end{defn}
\begin{rems}\label{Rmk: Manes critical value inf dim system}
  \item \textbf{Mañé's critical value for exact magnetic systems.} Given an exact magnetic system \((M,g,\mathrm{d}\alpha)\), we can associate with it a unique Lagrangian \(L\) using \eqref{magnetic Langrangian}. Thus, we can define the Mañé critical value of the magnetic system \((M,g,\mathrm{d}\alpha)\) through the Mañé critical value of \(L\), i.e.,
  \[
  c(M,g,\mathrm{d}\alpha) := c(L).
  \]
  
  \item \textbf{Adaptation to general magnetic systems.} \Cref{Def: Mañé's critical value or Hilber manifolds} can be easily adapted for general magnetic systems in \Cref{d: Magnetic system}, just as it is done for finite-dimensional magnetic systems in the references \cite{AssBenLust16,CIPP,  Man, Merry2010}. This adaptation will be carried out in an upcoming work by the author \cite{HopfRinowHalfLiegroups, MaierRuscelliTonelli}, along with some illustrations, but is left out at this point for the sake of simplicity.
\end{rems}

This value is intrinsically difficult to calculate for general magnetic systems. Furthermore, according to Contreras' result \cite{Co06}, for magnetic systems \((M,g,\mathrm{d}\alpha)\) on closed finite-dimensional Riemannian manifolds, there exists a magnetic geodesic connecting two given points \(p,q \in M\) with a prescribed kinetic energy \(k\) when \(k>c(L)\). This implies an analogue of the Hopf--Rinow theorem for magnetic geodesics of the system \((M,g,\mathrm{d}\alpha)\) with a kinetic energy exceeding Mañé's critical value. Note, however, that for energy values below Mañé's critical value, even in finite dimensions, the validity of a Hopf--Rinow theorem remains uncertain. 
Considering the infinite-dimensional counterexamples to the Hopf--Rinow theorem \cite{HopfrinowfalseAktkin, HopfrinowfalseEkeland} and the positive results on Banach half-Lie groups by Bauer-Harms-Michor \cite{Bauer_2025}, this raises a pertinent question:
\begin{ques}
   For a magnetic system \((M,g,\mathrm{d}\alpha)\) with $L$ Lagrangian and Mañé's critical value $c(L)$, given two points \(p,q \in M\), does there exist a magnetic geodesic with prescribed kinetic energy \(k\) if \(k > c(L)\)?
\end{ques}
In this paper, we provide an answer to the question posed for an exact infinite-dimensional magnetic system (see \Cref{t: Manes critical value for (M2HS)}). We prove even more: the Mañé critical value is supercritical in the following sense. For each energy value \( k \) below Mañé's critical value \( c(L) \), i.e., \( k \leq c(L) \), there exist at least two points \( p \) and \( q \) such that no magnetic geodesic with prescribed energy \( k \) connects \( p \) and \( q \). Conversely, for all energy values \( k > c(L) \) and for any points \( p \) and \( q \), there always exists a magnetic geodesic with prescribed energy \( k \) connecting \( p \) and \( q \).
Even for finite-dimensional magnetic systems on closed manifolds, it remains an open question whether Mañé's critical value is supercritical in the sense described above. In the following, we will reduce the infinite dimensional magnetic systems to finite-dimensional dynamical systems, which are much better understood. This reduction will be made precise through the notion of totally magnetic submanifolds.  
\subsection{Magnetomorphism and totally magnetic submanifolds}
	We will see that a key feature of the magnetic system \((G,g,\varPhi^*\mathrm{d}\alpha)\) is its large group of symmetries and, as a byproduct, a large family of invariant submanifolds. In \cite[§6]{ABM23}, Albers--Benedetti, and the author introduced a framework to handle these situations in general. For the convenience of the reader, let us recall the key definitions. The symmetries of a general magnetic system \((M,g,\sigma)\) are defined by so-called \emph{magnetomorphisms} \(F\colon M \to M\), which are diffeomorphisms that preserve both \(g\) and \(\sigma\). To be precise:

\begin{defn}\label{magnetic morphism and iso}{\cite[Def. 6.1]{ABM23}}\label{d: magnetomorphism}
	Let \((M_1,g_1,\sigma_1)\) and \((M_2,g_2,\sigma_2)\) be magnetic systems. A \emph{magnetomorphism} \(\Phi\colon M_1 \to M_2\) is a diffeomorphism such that
	\[
	\Phi^*g_2 = g_1 \quad \text{and} \quad \Phi^*\sigma_2 = \sigma_1.
	\]
	The \emph{group of magnetomorphisms} of a magnetic system \((M,g,\sigma)\) is defined as 
	\[
	\mathrm{Mag}(M,g,\sigma) := \left\{\Phi : M \to M \mid \Phi^*g = g, \ \Phi^*\sigma = \sigma\right\}.
	\]
\end{defn}
\begin{rem}{\cite[Rmk. 6.3.]{ABM23}}\label{r: magnetormorphism act on lorentz force}
	If $\Phi\colon(M_1,g_1,\sigma_1)\to(M_2,g_2,\sigma_2)$ is a magnetomorphism and $Y_1$ and $Y_2$ represent the Lorentz forces of the two magnetic systems, then
\begin{equation}\label{e:Ymag}
	\d\Phi\cdot Y^1=Y^2_\Phi\cdot\d\Phi.
\end{equation}
\end{rem}
As a consequence, magnetomorphisms map magnetic geodesics to magnetic geodesics with the same energy, see \cite[Prop. 6.4]{ABM23}. On the other hand, invariant submanifolds are characterized by so-called \emph{totally magnetic submanifolds} \(N \subset M\), which ensure that any magnetic geodesic tangent to \(N\) remains locally contained within \(N\). To elaborate:

\begin{defn}\label{Definition totally magnetic submanifold}{\cite[Def. 6.5]{ABM23}}
	Let \((M, g_M, \sigma_M)\) be a magnetic system. An embedded submanifold \(N \subset M\) is called totally magnetic if, for all magnetic geodesics \(\gamma \colon I \to M\) such that \(\gamma(0) \in N\) and \(\dot\gamma(0) \in T_{\gamma(0)}N\), there exists \(\varepsilon > 0\) such that \(\gamma((-\varepsilon, \varepsilon)) \subset N\).
\end{defn}

Thus, magnetomorphisms send totally magnetic submanifolds to totally magnetic submanifolds.
When \(\sigma = 0\), totally magnetic submanifolds revert to the classical notion of totally geodesic submanifolds. In Riemannian geometry, there are many ways to define a totally geodesic submanifold. In particular, a submanifold is totally geodesic if and only if its second fundamental form \(\mathrm{II}\) vanishes. To provide a similar statement for magnetic subsystems, we introduce the \emph{magnetic second fundamental form} \(\mathrm{II}^{\mathrm{mag}}\), which combines the second fundamental form \(\mathrm{II}\) with the difference between the intrinsic and extrinsic Lorentz forces. Precisely:

\begin{defn}\label{magnetic fundamental form}{\cite[Def. 6.9.]{ABM23}}
	The tensor
	\[
	\mathrm{II}^{\mathrm{mag}}_q \colon T_pN \to T_qN^{\perp}, \quad \mathrm{II}^{\mathrm{mag}}_q(v) := \mathrm{II}_q(v) + Y_q^Mv - Y_q^Nv \quad \forall q \in N,
	\]
	is called the \emph{magnetic second fundamental form} of the submanifold \(N\) within the magnetic system \((M, g_M, \sigma_M)\).
\end{defn}

Now, we can transform the local definition of totally magnetic submanifolds into an infinitesimal characterization and provide several equivalent definitions of totally magnetic submanifolds:
    	\begin{thm}\label{t:totmag}{\cite[Thm. 1.4]{ABM23}}
Let $(M,g_M,\s_M)$ be a magnetic system. Let $N\subset M$ be a closed, embedded submanifold. Denote by $\iota\colon N\to M$ the inclusion map and by $g_N:=\iota^*g_M$ and $\s_N:=\iota^*\s_M$ the pullback metric and magnetic field. The following statements are equivalent:\begin{enumerate}
			\item The submanifold $N$ is totally magnetic in $(M,g_M,\s_M)$.
			\item If $\gc$ is a magnetic geodesic in $\left(N,g_N,\s_N\right)$ then $\iota\circ\gc$ is a magnetic geodesic in $\left(M,g_M,\s_M\right)$.
			\item The magnetic second fundamental form of $\left(N,g_N,\s_N\right)$ vanishes identically: 
			\[
			\II^{\mathrm{mag}}_q(v)=0\quad \forall(q,v)\in TN.
			\]
			\item The second fundamental form of $(N,g_N)$ vanishes identically 
			\[
			\II_q(v)=0\quad \forall (q,v)\in TN.
			\]
			and one of the following three equivalent conditions holds
			\begin{enumerate}
				\item the Lorentz force $Y^M$ of $(M,g_M,\s_M)$ along $N$ is equal to the Lorentz force $Y^N$ of $(N,g_N,\s_N)$: 
				\[
				Y^N_qv= Y^M_qv\quad \forall (q,v)\in TN;
				\]
				\item the Lorentz force $Y^M$ leaves the tangent bundle of $N$ invariant: 
				\[
				Y_q^Mv\in TN \quad\forall(q,v)\in TN;
				\] 
				\item the $g_M$-orthogonal of $TN$ is contained in the $\s_M$-orthogonal of $TN$:
				\[(\sigma_M)_q(u,v)=0 \quad \forall q\in N,\ u \in T_q N,\ v\in T_q N^{\perp}. 
				\]
			\end{enumerate}
		\end{enumerate}
	\end{thm}

 \section{The magnetic system on \texorpdfstring{$\SiL$}{SiL}}\label{S: 3}
 In this section, we will investigate the magnetic system on an infinite-dimensional sphere in full depth. We compute its Mañé critical value and prove that it is supercritical in the sense described at the end of \Cref{Manes critical for Hilbert manifolds}. Consequently, we obtain a Hopf--Rinow type theorem for this infinite-dimensional magnetic system, which is contingent upon whether the prescribed energy is above Mañé's critical value. The theorem fails if the prescribed energy is below Mañé's critical value. To prove this, we provide a dynamical reduction theorem to a totally magnetic three-sphere. The magnetic system on \(\mathbb{S}^3\) is fully understood by \cite{ABM23}. We recall the following all results that are important to us:

  \subsection{Toy model: the magnetic system on $\Sd$}\label{subsection the toymodel Sd} Let us begin with introducing the setting in detail. Here the manifold $\Sd$ is the sphere of radius $1$ in $\C^{2}$ with standard Hermitian product $\langle\cdot,\cdot\rangle$, the metric  $g=\mathrm{Re}\langle\cdot,\cdot\rangle$ is the restriction of the Euclidean metric to $\Sd$, and the magnetic potential $\alpha$ is the standard contact form on $\Sd$, that is, $\alpha_z=\tfrac12\mathrm{Re}\langle \mathrm iz,\cdot\rangle$ for all $z\in\Sd$. The Reeb vector field of $\alpha$ is the unique vector field $R$ on $\Sd$ such that $\d\alpha(R,\cdot)=0$ and $\alpha(R)=1$. In this case, we get $R_z=2\mathrm iz$. The trajectories of the flow 
	\begin{equation}\label{e:hopfflow}
	\Phi_R^t(z)=e^{2\mathrm it}z,\qquad \forall\, t\in\RR,\ z\in\Sd
	\end{equation}
	are the fibers of the Hopf map $\pi\colon \Sd\to \mathbb CP^1$, which sends each point on $\Sd$ to the complex line through it. This is the simplest example of a Zoll Reeb flow, where all orbits of the Reeb vector field are periodic and with the same minimal period \cite{ABE24,APB}. Finally, $\ker\alpha$ is the contact distribution of $\alpha$ and, in our case, coincides with the orthogonal of $R$, for more detail on contact geometry we refere to \cite{Gg08}. Now we are in a position to state a special case of the main theorem in \cite{ABM23}, which computes the Mañé critical value of the magnetic system \((\mathbb{S}^3, g, \mathrm{d}\alpha)\) and proves that it is supercritical: 
    	\begin{prop}\label{t:mane}{\cite[Thm. 1.1]{ABM23}}
	The Mañé's critical value of the system is
	\[
	c(\Sd,g,\d\alpha)=\tfrac{1}{2}\Vert\alpha\Vert_\infty^2=\tfrac18.
	\]
	Let $q_0$ and $q_1$ be two points on $\Sd$ and denote by $\langle q_0,q_1\rangle$ their Hermitian product. For every $k>0$, let $\mathcal G_k(q_0,q_1)$ be the set of magnetic geodesics with energy $k$ connecting $q_0$ and $q_1$. We have the following three cases
	\begin{enumerate}
		\item if $k>\tfrac18$, then $\mathcal G_k(q_0,q_1)\neq\varnothing$;
		\item if $k=\tfrac18$, then $\mathcal G_k(q_0,q_1)\neq\varnothing$ if and only if $\langle q_0,q_1\rangle\neq0$;
		\item if $k<\tfrac18$, we have the following three subcases 
		\begin{enumerate}
			\item if $|\langle q_0,q_1\rangle|> \sqrt{1-8k}$, then $\mathcal G_k(q_0,q_1)\neq\varnothing$;
			\item if $|\langle q_0,q_1\rangle|=\sqrt{1-8k}$, then there are $a_k,b_k\in\RR$ with $b_k>0$ such that $\mathcal G_k(q_0,q_1)\neq\varnothing$ if and only if $\langle q_0,q_1\rangle=e^{\mathrm i(a_k+mb_k)}\sqrt{1-8k}$ for some $m\in\ZZ$;
			\item if $|\langle q_0,q_1\rangle|< \sqrt{1-8k}$, then $\mathcal G_k(q_0,q_1)=\varnothing$.
		\end{enumerate}
	\end{enumerate}
	\end{prop}
    \begin{rems} 
     Indeed, the Mañé critical value is supercritical in the sense of the end of  \Cref{Manes critical for Hilbert manifolds}, i.e., the magnetic system \((M,g,\mathrm{d}\alpha)\) satisfies a Hopf--Rinow theorem if and only if the energy value is greater than the Mañé critical value.
    \end{rems}
   Before we move on, let us mention that the magnetic system \((\mathbb{S}^3, g, \mathrm{d}\alpha)\) is particularly complex because it encapsulates all magnetic systems on \(\mathbb{S}^2\) at once, as shown in \cite[Thm. 1.12, Cor. 1.14]{ABM23}. This makes it significantly more complicated than the classical magnetic system on the round two-sphere. The implications of this complexity and the meaning of \Cref{t:mane} are illustrated in \Cref{magnetic flow S3-> S2}.

 	\begin{figure}[htbp]
		
		\includegraphics[width=0.7\textwidth]{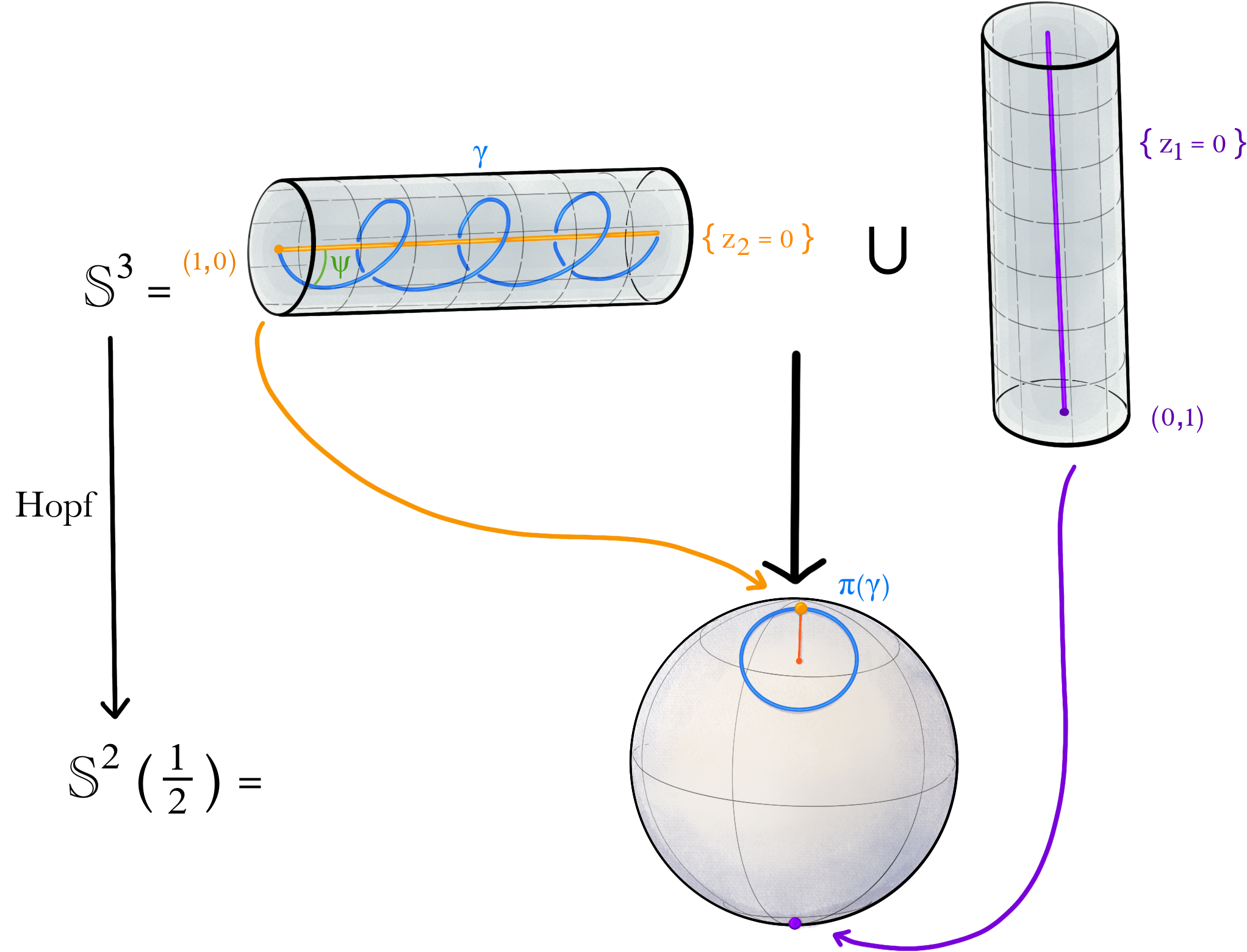}
		\caption{This picture illustrates \cite[Thm. 1.12]{ABM23} and the fact that we can connect the Reeb orbit \textcolor{orange}{$\{z_2=0\}$} and the Reeb orbit \textcolor{purple}{$\{z_1=0\}$} with a unit speed magnetic geodesic of prescribed energy $k$ if and only if $k>\frac{1}{8}$, see Theorem \ref{t:mane}. (Picture by Ana Chavez Caliz.)}
		\label{magnetic flow S3-> S2}
	\end{figure}
\subsection{The magnetic system $(\SiL, \GL, \d\alpha)$}
In this subsection, the magnetic system $(\Sd, g, \d\alpha)$ is generalized to an infinite-dimensional setting, and a dynamical reduction theorem is provided in the sense of totally magnetic submanifolds. This allows for the reduction of the dynamics of the infinite-dimensional system to $(\Sd, g, \d\alpha)$. The setting is introduced in detail as follows: The construction of the $(\Sd, g, \d\alpha)$ magnetic system relies on the fact that $\Sd$ is a hypersurface in the ambient space $\CC^2$, which constitutes a complex Hilbert space. Leveraging the Hilbert space structure of $\CC^2$, one can define the round metric and the standard contact form on $\Sd$, as detailed in \cref{subsection the toymodel Sd}.

Consequently, it is logical to replace the Hilbert space $\CC^2$ with the classical Hilbert space $L^2(\SS^1, \CC)$, consisting of square integrable complex-valued functions. This space is equipped with the standard $L^2$ product $\langle \cdot, \cdot \rangle_{L^2}$, which induces the $L^2$ norm denoted by $\Vert \cdot \Vert_{L^2}$. The sphere of radius one in $L^2(\SS^1, \CC)$ is denoted by $\SiL$, defined as
\[
    \SiL:= \left\{f \in L^2(\SS^1, \CC) : \Vert f \Vert_{L^2} = 1\right\}.
\]
Again, the round metric $g = \Re \langle \cdot, \cdot \rangle$ is the restriction of the Euclidean metric on $L^2(\SS^1, \CC)$ to $\SiL$. Before defining the magnetic potential, it should be noted that, as in the finite-dimensional setting, the tangent space of $\SiL$ at any point $f$ is given by
\[
T_f\SiL = \left\{ F \in L^2(\SS^1, \CC) : \Re\langle f, F \rangle_{L^2} = 0 \right\}.
\]
Utilizing the Hilbert space structure of $L^2(\SS^1, \CC)$, the magnetic potential is defined similarly to the case of the toy model:
\begin{defn}\label{def:contact_form_infinite_dimension}
    The standard contact form $\alpha$ on $\SiL$ is defined as the linear functional: 
    \[
    \alpha_{f} : T_{f}\SiL \longrightarrow \RR : F \mapsto \frac{1}{2} \Re\langle i f, F \rangle_{L^2} \quad \forall (f, F) \in T\SiL.
    \]
\end{defn}
\begin{rem}
   By its definition $\alpha_{f}$ is a bounded linear operator in each point $f\in \SiL$, thus it is indeed a one-form on $\SiL$. For more details on differential forms on Hilbert manifolds, we refer to \cite{la99}.
\end{rem}
\begin{rem}
    To the best of the authors knowledge, there is no uniform definition of a contact form in the infinite-dimensional setting. We use a naive approach by defining a tangent distribution \(\xi := \ker \alpha\) and demonstrate that this distribution behaves similarly to that in the finite-dimensional setting. Alternatively, one could use the approach involving a so-called Liouville vector field, as described by Abbondandolo-Majer \cite{AM15}.
\end{rem}
Note before we move on that as in the finite dimensional setting, the derivative of the standard contact form on $\SiL$ is equal to the restriction of $\Im\langle \cdot, \cdot \rangle$ to the sphere, here $\SiL$. There exist a unique vector field $R$ defined through $\d\alpha(R, \cdot)=0$ and $\alpha(R)=1$, called the \emph{Reeb vector field} of $\alpha$. It's an exercise left to the reader that in this case we get $R_f=2\i f$. The trajectories of the flow of $R$
\[
\varPhi^t_R(f)= e^{2\i t} \cdot f\quad t\in \RR,\ f\in\SiL
\]
are the fibers of the infinite dimensional Hopfmap $\pi:\SiL\longrightarrow \CP$, which sends each point on $\SiL$ to the complex line through it. As in the finite dimensional setting we define
\begin{defn}\label{standard comtact structure on SiL}
 The distribution $\xi:=\ker \alpha$ is called the \emph{standard contact structure} on $\SiL$.
\end{defn}
In order to see that this definition is well-defined we have to prove 
\begin{prop}\label{p: contact distri on SIL has codim one}
    The contact distribution has codimension one in $T\SiL$ as a bundle and $T\SiL$ possesses the orthogonal splitting \[
    T\SiL= \xi \oplus \langle R \rangle_{\RR}.
    \]
\end{prop}
\begin{proof}
    By its definition the image $\mathrm{im} (\alpha_f)$ is of dimension one at each $f\in \SiL$, since $\alpha (R)=1$. So $\alpha_f$ is surjective at each $f\in \SiL$ and since $\xi=\ker\alpha$ is closed by definition so by the fundamental homomorphism theorem in the category of Hilbert spaces we get the following diagram 
    \begin{align*}
			\xymatrix{
				T_{f}\SiL\ar@{->>}[dr]\ar@{->>}[rr]
				^{\alpha_{f}} &   &\RR \\
				& T_{f}\SiL/\xi_{f}\ar@{.>}[ru]
				^{\cong} & }
		\end{align*} 
        commutes for all $f\in \SiL$, which finishes the proof.
\end{proof}
We now aim to describe the magnetic system \((\SiL, \GL, \mathrm{d}\alpha)\) in more detail. Following the approach of \cite[§3]{ABM23}, we can derive the magnetic geodesic equation of \((\SiL, \GL, \mathrm{d}\alpha)\). As in \cite{ABM23}, the Levi-Civita connection evaluated along a unit speed magnetic geodesic is given by 
\begin{equation}\label{e:Levi cevita on SIl}
    \nabla_{\dot{\gamma}}\dot{\gamma}=\ddot{\gamma}+\gamma.
\end{equation}
On the other hand, the Lorentz force is explicitly given by a rotation of 90 degrees of the orthogonal projection \(\pi_{\xi}\) of the velocity to the contact distribution, i.e.,
\begin{equation}\label{e:Lorenz force on SIl}
    Y_f(F)= \i \pi_{\xi}(h)= \i \cdot\left(F- \Re\langle \i f, F\rangle_{L^2} \cdot\i f\right) \quad \forall (f,F) \in T\SiL.
\end{equation}
 By \eqref{e:mg}, \eqref{e:Levi cevita on SIl}, and \eqref{e:Lorenz force on SIl}  specifically, a curve \(\gamma\) is a magnetic geodesic on \((\SiL, \GL, \mathrm{d}\alpha)\) with unit speed and of magnetic strength \(s\) if and only if \(\gamma\) satisfies
\begin{equation}\label{e:magnetic geodesic equation in SIL}
    \ddot{\gamma} - \i s \dot{\gamma} + \left(1 - s \cos \psi \right) \gamma = 0,
\end{equation}
where \(\cos \psi\) is a conserved quantity representing the angle between \(\dot{\gamma}\) and \(R\), given by
\begin{equation}\label{e: contact angle}
\cos \psi = \operatorname{Re} \langle  \i\gamma ,\dot{\gamma}\rangle_{L^2}.
\end{equation}

\begin{rem}
    The derivation of the magnetic geodesic equation in $\SiL$, as referred to in \eqref{e:magnetic geodesic equation in SIL}, is left as an exercise for the interested reader, note that it suffices to replace \(\S\) with \(\SiL\) in \cite[§3]{ABM23}.
\end{rem}
 \subsection{Dynamical reduction from $\Sil$ to $\Sd$} The aim of this subsection is to provide a dynamical reduction theorem, i.e., the dynamics of the magnetic system $\left(\SiL, \GL, \d\alpha \right)$ can be reduced to the magnetic system $\left(\Sd, g, \d\alpha \right)$. More precisely, we prove that the dynamics of the infinite-dimensional sphere can always be reduced to a totally magnetic three-dimensional sphere. This result represents a special case of a much more general perspective, as it follows directly from the classification of totally magnetic submanifolds of the magnetic system $\left(\SiL, \GL, \d\alpha \right)$:

\begin{prop}\label{prop: Classification of totally magnetic submanifolds of SiL}
    The group of magnetomorphisms of $\left(\SiL, \GL, \d\alpha \right)$ is the group $\mathcal{U}_{L^2}$ of unitary operators of the Hilbert space $L^2\left(\SS^1, \CC \right)$, that is, 
    \[
    \mathrm{Mag}\left( \SiL, \GL, \d\alpha \right) = \mathcal{U}_{L^2}.
    \]
    The lift of the magnetomorphism group to $T\SiL$ is Hamiltonian with respect to the twisted symplectic form $\omega_{\d\alpha}$, defined as in \eqref{e: twisted symplectic form}. Its moment map is given by
\[
\mu\colon T\SiL\to \mathfrak u^*,\qquad \mu(f,F)[A]:= \GL_f(Af,F)-\alpha_f(Af),\quad \forall\,A\in\mathfrak u,\ \forall\,(f,F)\in T\SiL,
\]
where $\mathfrak{u}^*$ denotes the dual of the Lie Algebra $\mathfrak{u}:=\mathrm{Lie} (\mathcal{U}_{L^2})$\\
    A closed connected submanifold $N$ of $\SiL$ with positive dimension is totally magnetic if and only if $N = \SiL \cap V$, where $V$ is a complex linear subspace of $\Lt$. In this case, $N$ is magnetomorphic to the unit sphere $\SS(V)$ of the Hilbert space $V$, that is,
    \[
    \left(N, \GL, \d\alpha \right) \cong \left(\SS(V), \GL, \d\alpha \right).
    \]
\end{prop}

\begin{proof}
    The proof follows directly along the lines of \cite[Cor. 1.4]{ABM23}, and the details are left to the interested reader.
\end{proof}

\begin{rem}
    \Cref{prop: Classification of totally magnetic submanifolds of SiL} illustrates the difference between totally magnetic submanifolds and totally geodesic submanifolds. Indeed, every unit sphere $\SS(H)$ of a real subspace $H \subset \Lt$ is totally geodesic in $\SiL$, while only spheres obtained from the intersection of $\SiL$ with complex subspaces $V$ of $\Lt$ lead to totally magnetic submanifolds.
\end{rem}

As a direct corollary of \Cref{prop: Classification of totally magnetic submanifolds of SiL}, we obtain the aforementioned dynamical reduction theorem:

\begin{cor}\label{cor: dynamical reduction of SiL to Sd}
	Each magnetic geodesic in \(\left(\SiL, \GL, \d\alpha \right)\) with initial values \((f, F) \in T\SiL\) is contained in a three-dimensional sphere \(\Sd(f, F)\), obtained by intersecting \(\SiL\) with the complex plane \(\mathrm{span}_{\CC}\{f, F\}\).
\end{cor}

With this dynamical reduction theorem in mind, it should be no surprise that we can generalize the interpolation property from the magnetic geodesic flow on \((\Sd, g, \d\alpha)\) to the infinite-dimensional setting. Recall that the contact angle \(\psi\) (\eqref{e: contact angle}) is a conserved quantity. For every \(s \in \RR\), the magnetic geodesics of strength \(s\) with \(\psi = 0\), respectively \(\psi = \pi\), are exactly the orbits of the Reeb vector field with the positive, respectively negative, orientation. Magnetic geodesics with \(\psi = \tfrac{\pi}{2}\) are tangent to the contact distribution \(\ker\alpha\).

As a consequence of \Cref{cor: dynamical reduction of SiL to Sd} and \cite[Cor. 1.8]{ABM23}, we obtain:

\begin{cor}\label{C: interpolation}
	The magnetic geodesic flow \((\SiL, \GL, \d\alpha)\) is an interpolation between the sub-Riemannian geodesic flow on \((\SiL, \GL, \ker\alpha)\) and the Reeb flow of \(\alpha\) on \(\SiL\), where the interpolation parameter is the contact angle \(\psi\).\hfill\qed
\end{cor}

Before moving on, we illustrate how we can use the dynamical reduction in the sense of \Cref{cor: dynamical reduction of SiL to Sd} to show that the magnetic system \(\left(\SiL, \GL, \d\alpha\right)\) possesses infinitely many conserved quantities:

\begin{cor}\label{c: magnetic system SiL poss infty many integrals}
	The magnetic system \(\left(\SiL, \GL, \d\alpha\right)\) possesses infinitely many conserved quantities.
\end{cor}

\begin{proof}
	By \Cref{cor: dynamical reduction of SiL to Sd}, we can reduce the magnetic system to \(\left(\Sd, g, \d\alpha \right)\). This reduced system is, by \cite[Appendix A]{BM24}, totally integrable. In addition, the orthogonal projection of the flow of the magnetic system on \(T \SiL\) to \(T\Sd^{\perp}\) provides infinitely many integrals of motion, as \(T\Sd^{\perp}\) is invariant under the magnetic flow, after choosing initial conditions in $T\Sd$,  and is infinite-dimensional.
\end{proof}

In the next subsection, we illustrate the importance of \Cref{cor: dynamical reduction of SiL to Sd} in providing a generalization of \Cref{t:mane} to the infinite-dimensional magnetic system \(\left(\SiL, \GL, \d\alpha \right)\).

\subsection{Mané's Critical Value and a Hopf--Rinow Theorem for \(\left(\SiL, \GL, \d\alpha\right)\)}

In this subsection, we compute Mané's critical value for the magnetic system \(\left(\SiL, \GL, \d\alpha\right)\) and illustrate how it can be used to obtain a Hopf--Rinow theorem for \(\left(\SiL, \GL, \d\alpha\right)\) as a generalization of \Cref{t:mane}. 

We begin by computing Mané's critical value for the magnetic system \(\left(\SiL, \GL, \d\alpha\right)\):

\begin{lem}\label{L: Mane crit value for SIL}
	Mané's critical value for the magnetic system is 
	\[
	c\left(\SiL, \GL, \d\alpha\right) = \frac{1}{8}.
	\]
\end{lem}

\begin{proof}
	By the definition of Mané's critical value (\Cref{Def: Mañé's critical value or Hilber manifolds}), it suffices to prove that, first, one can always find a periodic orbit with negative action \( S_{L+k} \) and prescribed energy below \( \frac{1}{8} \), and second, the action \( S_{L+k} \) is positive for all energies above \( \frac{1}{8} \).
	The proof proceeds by verifying these two conditions. First, one can verify that for all \( k < \frac{1}{8} \), the curve \(\gamma(t) = \mathrm{id}_{\SS^1} \cdot e^{\i t \sqrt{k}}\) is a periodic orbit with negative action. Additionally, we can bound the magnetic Lagrangian from below as follows:
	\[
	L(f, F) + \frac{1}{8} = \frac{1}{2} \Re\langle F, F \rangle_{L^2} - \frac{1}{2} \Re\langle \i f, F \rangle_{L^2} + \frac{1}{8} \geq \frac{1}{2} \left( \Vert F \Vert_{L^2}^2 - \frac{1}{2} \right)^2 \geq 0.
	\]
	Thus, \( S_{L+k} \geq 0 \) for all \( k \geq \frac{1}{8} \), which finishes the proof.
\end{proof}

Now we are in the position to generalize \Cref{t:mane} to the infinite-dimensional setting as a combination of \Cref{cor: dynamical reduction of SiL to Sd} and \Cref{t:mane}. That is, we will prove an infinite-dimensional Hopf--Rinow theorem for magnetic flows.

\begin{thm}\label{t:mane for SiL}
	Mané's critical value of the system is
	\[
	c(\SiL, \GL, \d\alpha) = \tfrac{1}{2}\Vert\alpha\Vert_\infty^2 = \tfrac{1}{8}.
	\]
	Let \(q_0\) and \(q_1\) be two points on \(\SiL\) and denote by \(\langle q_0,q_1\rangle_{L^2}\) their \(L^2\)-Hermitian product. For every \(k>0\), let \(\mathcal{G}_k(q_0,q_1)\) be the set of magnetic geodesics with energy \(k\) connecting \(q_0\) and \(q_1\). We have the following three cases:
	\begin{enumerate}
		\item\label{it: 1 Mane for SiL} If \(k > \tfrac{1}{8}\), then \(\mathcal{G}_k(q_0,q_1) \neq \varnothing\).
		\item\label{it: 2 Mane for SiL}  If \(k = \tfrac{1}{8}\), then \(\mathcal{G}_k(q_0,q_1) \neq \varnothing\) if and only if \(\langle q_0,q_1\rangle \neq 0\).
		\item\label{it: 3 Mane for SiL}  If \(k < \tfrac{1}{8}\), we have the following three subcases:
		\begin{enumerate}
			\item\label{it: 3a Mane for SiL}  If \(|\langle q_0,q_1\rangle| > \sqrt{1-8k}\), then \(\mathcal{G}_k(q_0,q_1) \neq \varnothing\).
			\item\label{it: 3b Mane for SiL}  If \(|\langle q_0,q_1\rangle| = \sqrt{1-8k}\), then there are \(a_k,b_k \in \RR\) with \(b_k > 0\) such that \(\mathcal{G}_k(q_0,q_1) \neq \varnothing\) if and only if \(\langle q_0,q_1\rangle = e^{\mathrm{i}(a_k + mb_k)}\sqrt{1-8k}\) for some \(m \in \ZZ\).
			\item\label{it: 3c Mane for SiL}  If \(|\langle q_0,q_1\rangle| < \sqrt{1-8k}\), then \(\mathcal{G}_k(q_0,q_1) = \varnothing\).
		\end{enumerate}
	\end{enumerate}
\end{thm}

\begin{rem}
	This theorem allows us to compare the magnetic geodesic flow and the geodesic flow by means of a Hopf--Rinow theorem, showing that a value of the energy is supercritical if and only if all pairs of points on the sphere can be connected by a magnetic geodesic with that value of energy, where Mañé's critical value serves as a threshold.
\end{rem}

\begin{proof}
	This follows directly from \Cref{t:mane}, \Cref{cor: dynamical reduction of SiL to Sd}, and \Cref{L: Mane crit value for SIL}.
\end{proof}

To continue our discussion, it is convenient to switch to the description of magnetic geodesics of \((\SiL, \GL, \d\alpha)\) with energy \(k\) as magnetic geodesics of unit speed for the family of systems \((\SiL, \GL, s\d\alpha)\), where \(s = \frac{1}{\sqrt{2k}}\) is now referred to as the strength of the magnetic geodesic. In this description, Mané's critical value corresponds to
\[
s_0 := \frac{1}{\sqrt{2\tfrac{1}{8}}} = 2.
\]

Before proceeding, let us mention that, as a corollary of \Cref{cor: dynamical reduction of SiL to Sd}, we can generalize the results of \cite[Thm. 1.12, Cor. 1.14]{ABM23} to the infinite-dimensional setting. In other words, one can recover all magnetic systems on \(\mathbb{C}P^{\infty}\) simultaneously through the magnetic system on \(\SiL\). Here, \(\mathbb{C}P^{\infty}\) denotes the infinite-dimensional projective space, the projectivization of \(\Lt\), equipped with the Fubini-Study form \(\omega^{FS}\) and the Fubini-Study metric \(\mathcal{G}^{FS}\); this serves as the prototype of an infinite-dimensional Kähler manifold (see~\cite[§2]{KMM19}). 

\begin{cor}\label{C:hopf}
	If \(\gamma\) is a unit-speed magnetic geodesic on \((\SiL, \GL, s \, \d\alpha)\), \(s \geq 0\), with contact angle \(\psi \in [0, \pi]\), then, up to orientation-preserving reparametrization, \(\pi(\gamma)\) is a unit-speed magnetic geodesic of \(\big(\mathbb{C}P^{\infty}, \mathcal{G}^{FS}, a_s(\psi)\omega^{FS}\big)\), where 
	\begin{equation}\label{e:curvature}
		a_s(\psi) = \frac{2\cos\psi - s}{\sin\psi}, \qquad \psi \in (0, \pi).
	\end{equation}
	In particular, the projected curve \(\pi(\gamma)\) is a geodesic circle with radius 	
	\[
	r_s(\psi) := \frac{1}{a_s(\psi)} \in \left[0, \frac{\pi}{2}\right].
	\]
\end{cor}

We observe that the function \(a_s\) is a bijection if and only if \(0 \leq s < 2\). This fact can be summarized in the following statement:

\begin{cor}\label{C Sil onto CP inf}
	For a fixed \(s \in [0, 2)\), the unit-speed magnetic geodesic flow on \((\SiL, \GL, s \, \d\alpha)\) covers \emph{all} magnetic systems \(\big(\mathbb{C}P^{\infty}, \mathcal{G}^{FS}, r\omega^{FS}\big)\), \(r \, \in \mathbb{R}\).
\end{cor}

\begin{rem}
	Here, Mané's critical value provides a natural explanation for the threshold of the magnetic strength in \Cref{C Sil onto CP inf}. 
\end{rem}

\section{The Madelung transform as magnetomorphism}
The aim of this section is to prove that we extend the \emph{Madelung transform} as an isometry between an infinite-dimensional Lie group and the open subset of nowhere-vanishing functions on the \(L^2\)-sphere to a magnetomorphism. 

We begin by introducing the setting in detail. Let \(\SS^1 = \RR / \ZZ\) denote the unit circle, and let \(\mathrm{Diff}_0^m\left(\SS^1\right)\) represent the half-Lie group of diffeomorphisms of $\SS^1$ of Sobolev with smooth right multiplication and only continuous left multiplication; see \cite{Bauer_2025} for the notion of a half-Lie group. This group consists of all diffeomorphisms of \(\SS^1\) of Sobolev class \(H^m\) that fix a designated point on \(\SS^1\). Unless stated otherwise, we assume \(m > \frac{5}{2}\). Note that \(\mathrm{Diff}_0^m(\SS^1) \cong \mathrm{Dens}^m(\SS^1)\), where \(\mathrm{Dens}^m(\SS^1)\) denotes the space of probability densities on \(\SS^1\) of Sobolev class \(H^m\), as discussed in \cite[§5]{EM70} and \cite{km02}. 
We further denote by \(\SS^1_{4\pi}\) the circle of length \(4\pi\), and by \(H^m\left(\SS^1, \SS^1_{4\pi}\right)\) the space of maps of Sobolev class \(H^m\). The half-Lie group \(G^m\) is then defined as 
\begin{equation}\label{e: defi of Gm}
    G^m = \mathrm{Diff}_0^m\left(\SS^1\right) \rtimes H^{m-1}\left(\SS^1, \SS^1_{4\pi}\right),
\end{equation}
where the explicit group product structure is described in \cite[§2]{Lennels13}. This group carries a bi-invariant metric, the $\dot{H}^1$-metric, defined at the identity \((\id, 0)\) by
\begin{equation}\label{e: defi H1 dor metric at identity}
\mathcal{G}^{\dot{H}^1}_{(\id, 0)}\left((u, \rho), (v, \psi)\right) :=\langle (u, \rho), (v, \psi)\rangle_{(\id, 0)}^{\dot{H}^1} = \frac{1}{4} \int_{\SS^1} u_x v_x + \rho \psi \, \mathrm{d}x,
\end{equation}
where the tangent space at the identity is \(T_{(\id, 0)}G^m \cong H^m_0\left(\SS^1, \RR\right) \times H^{m-1}\left(\SS^1, \RR\right)\). Here, \(H^m_0\left(\SS^1, \RR\right)\) denotes the space of all functions of Sobolev class \(H^m\) with zero mean. We can extend the metric in \eqref{e: defi H1 dor metric at identity} by right invariance to all points \((\varphi, \tau) \in G^m\) by 
\begin{equation}\label{e: definition H1 dot metric at arbitrary element}
	\mathcal{G}^{\dot{H}^1}_{(\varphi, \tau)}\left((U_1, U_2), (V_1, V_2) \right) := \langle (U_1, U_2), (V_1, V_2)\rangle_{(\varphi, \tau)}^{\dot{H}^1} = \frac{1}{4} \int_{\SS^1} \frac{U_{1x}V_{1x}}{\varphi_x} + U_2 V_2 \varphi_x \, \mathrm{d}x,
\end{equation}
where \((U_1, U_2), (V_1, V_2) \in T_{(\varphi, \tau)}G^m \cong H^m_0\left(\SS^1, \RR\right) \times H^{m-1}\left(\SS^1, \RR\right)\). For more details, we refer again to \cite[§2]{Lennels13}. 

Furthermore, we denote by \(\SiLm\) the \(L^2\)-unit sphere of all functions of Sobolev class \(H^m\), and the open subset of nowhere-vanishing functions in \(\SiLm\) by 
\[
\U^m := \left\{f \in \SiLm : f(x) \neq 0, \quad \forall x \in \SS^1 \right\},
\]
which naturally carries, as an open subset of \(\SiL\), a Riemannian metric obtained by restricting the round metric \(\mathcal{G}^{L^2}(\cdot, \cdot) := \Re \langle \cdot , \cdot \rangle_{L^2}\) to \(U^m\).

Lennels proved in \cite[Thm. 3.1]{Lennels13} that the Madelung transform \(\varPhi\) is an isometry between \(\left( G^m, \mathcal{G}^{\dot{H}^1}\right)\) and \(\left(\U^{m-1}, \GL \right)\), i.e., 
\begin{equation}\label{e: Madelung transform is isometry}
	\varPhi: \left( G^m, \mathcal{G}^{\dot{H}^1}\right) \longrightarrow \left(\U^{m-1}, \GL \right) : (\varphi, \tau) \mapsto \sqrt{\varphi_x} \ e^{\i \frac{\tau}{2}}
\end{equation}
is an isometry for all \(m > \frac{5}{2}\).  By restricting the magnetic system \(\left( \SiL, \GL, \d\alpha\right)\) to the open subset \(U^{m}\), we naturally obtain, by a slight abuse of notation, the magnetic system \(\left( U^{m}, \GL, \d\alpha\right)\). 
By pulling back the magnetic field \(\d\alpha\) with the Madelung transform \(\varPhi\) (see \eqref{e: Madelung transform is isometry}), we naturally obtain the magnetic system \(\left( G^m, \GH, \varPhi^*\d\alpha\right)\).
Next, we prove that the Madelung transform is indeed a magnetomorphism: 
\begin{thm}\label{t: Madelungtransform as Magnetomorphims}
	The Madelung transform \(\varPhi\), see \eqref{e: Madelung transform is isometry}, is a magnetomorphism, in the sense of \Cref{d: magnetomorphism}, between the magnetic systems \(\left( G^m, \mathcal{G}^{\dot{H}^1}, \varPhi^*\d\alpha\right) \) and  \(\left( \U^{m-1}, \GL, \d\alpha\right)\) for all \(m > \frac{5}{2}\).
\end{thm}
\begin{rem}
	In particular, \Cref{t: Madelungtransform as Magnetomorphims} also proves that the Madelung transform \(\varPhi\) defines a magnetomorphism between the two above-mentioned magnetic systems in the category of inverse limit Hilbert manifolds or tame Fréchet manifolds. See \cite{EM70} and \cite{Khesin-Mis-Mod-inf-Newton}, along with the references therein, for an explanation of these notions of smoothness.
\end{rem}
\begin{proof}
	Follows directly from \Cref{d: magnetomorphism} in combination with the fact that the Madelung transform, \eqref{e: Madelung transform is isometry}, is an isometry, as established by \cite[Thm. 3.1]{Lennels13}.
\end{proof}
Before moving on, we note that the magnetic system \(\left( G^m, \mathcal{G}^{\dot{H}^1}, \varPhi^*\d\alpha\right) \) arises naturally through the lens of infinite-dimensional differential geometry. Here, we work in the smooth category and denote by \( G \) the \( C^{\infty} \) version of \( G^m \) in \eqref{e: defi of Gm}. This is a regular Lie group, following the notation in \cite{KrieglMichor1997}.
First, the infinite-dimensional Hopf fibration \(\pi: \SiL \longrightarrow \mathbb{C}P^{\infty}\) is, by definition, a so-called \emph{Boothby-Wang bundle} or \emph{prequantization bundle} in the sense of \cite[§7]{Gg08}. Here, \(\mathbb{C}P^{\infty}\) denotes the infinite-dimensional projective space, the projectivization of \(\Lt\), equipped with the Fubini-Study form \(\omega^{FS}\); see \cite[§2]{KMM19} for more details. 
If we denote by \(\mathbb{P}U\) the projectivization of \(U\), the restriction of the Hopf fibration to \(U\) also defines a Boothby-Wang bundle. Note that the magnetic system \(\left( \mathbb{P}U, \mathcal{G}^{FS}, \omega^{FS} \right)\) is, as a special case of the beautiful construction by Khesin-Misiolek-Modin in \cite{KhesinMisilokeModinMadelung-proc, KMM19}, Kähler isomorphic to the cotangent bundle of the space of probability densities \(T^*\mathrm{Dens}(\SS^1)\), equipped with the so-called Sasaki-Fisher-Rao metric \(\mathcal{G}^{\text{Rao}}\) and the canonical symplectic form \(\d\lambda\). This Kähler isomorphism is exactly the Madelung transform. Again, we refer to \cite{KMM19} for an explanation.  In summary, we obtain the following:

\begin{cor}\label{C: TDens as boothby wang bundel}
The tame Fréchet Lie group \((G, \varPhi^*\alpha)\) equipped with its natural contact form is a Boothby-Wang bundle over the cotangent bundle of the space of probability densities \((T^*\mathrm{Dens}(\SS^1), \d\lambda)\) equipped with the canonical symplectic form in the category of tame Fréchet manifolds. In particular, we obtain the commutative diagram, here, \(\cong\) indicates magnetomorphic: 
\[
\begin{tikzcd}[row sep=large, column sep=large]
	\left(G, \GH, \Phi^* \d\alpha\right) 
	\arrow[rr, "\cong", "\Phi = \text{Madelung}"'] 
	\arrow[dd] & &
	\left(U, \GL, \d\alpha\right) \subseteq \left(S^\infty, \GL, \d\alpha\right) 
	\arrow[dd, "\text{Hopf}"] \\
	& & \\
	\left(T^*\mathrm{Dens}(\SS^1), \mathcal{G}^{\text{Rao}}, \d\lambda\right) 
	\arrow[rr, "\cong", "\Phi = \text{Madelung}"'] & &
	\left(\mathbb{P}U,  \mathcal{G}^{FS}, \omega^{FS}\right) \subseteq \left(\mathbb{C}P^\infty, \mathcal{G}^{FS}, \omega^{FS}\right)
\end{tikzcd}.
\]

\end{cor}
\begin{rem}\label{r: all magnetic systems on TDens can recoverd through G GH dalp}
	Using the fact that the diagram in \Cref{C: TDens as boothby wang bundel} is commutative and that \(\varPhi\) is a magnetomorphism, we can apply \Cref{C Sil onto CP inf} to deduce that, for a fixed \(s \in [0, 2)\), the unit-speed magnetic geodesic flow on \(\left(G, \GH,s\cdot \Phi^* \d\alpha\right)\) covers \emph{all} magnetic systems \(\left(T^*\mathrm{Dens}(\SS^1), \mathcal{G}^{\text{Rao}}, r\cdot\d\lambda\right)\), \(r \in \mathbb{R}\).
\end{rem}

\section{The (M2HS) as a magnetic geodesic equation}
The aim of this section is to establish that the solutions of the highly nonlinear system of partial differential equations introduced as the \textbf{magnetic two-component Hunter--Saxton system}~\eqref{(M2HS)} are in one-to-one correspondence with the magnetic geodesics of the magnetic system \(\left(G, \GH, \Phi^* \mathrm{d}\alpha\right)\). Before going into detail, we provide a brief overview of this section.

In \Cref{ss: 5 the equations M2HS}, we state the main theorem of this section, namely the above-mentioned duality result. The subsequent sections, \Cref{ss: 5.2 the lorentz force}--\Cref{ss: 5.4 proof of duality}, are devoted to the proof of \Cref{t: magnetic Hunter Saxton system}. In \Cref{ss: 5.2 the lorentz force}, we derive the Lorentz force associated with the magnetic system appearing in \Cref{t: magnetic Hunter Saxton system}. In \Cref{ss: 5.3 derivation of mag geode eq}, we derive the corresponding magnetic geodesic equation. Finally, in \Cref{ss: 5.4 proof of duality}, we complete the proof of \Cref{t: magnetic Hunter Saxton system} by combining the results of the previous sections.\\
We conclude the section in \Cref{ss:infinitely many conserved quantities} by proving a statement of independent interest, namely that \eqref{(M2HS)} possesses infinitely many conserved quantities.
\subsection{Statement of the Duality Theorem}\label{ss: 5 the equations M2HS}
To prepare for this, we introduce the following notation. For a curve in \(\Diff_0(\mathbb{S}^1)\), we denote its time and spatial derivatives by \(\varphi_t\) and \(\varphi_x\), respectively. The same convention applies to vector fields \(u\), where \(\varphi\) is replaced by \(u\).

\begin{thm}\label{t: magnetic Hunter Saxton system}
	A \(C^2\)-curve \((\varphi,\tau): [0,T)\longrightarrow G^m\), where \(T>0\) is the maximal existence time, is a magnetic geodesic of the system \(\left(G^{m}, \GH, \varPhi^{*}\mathrm{d}\alpha\right)\) with \(m > \frac{5}{2}\) if and only if  
	\[
	(u = \varphi_t \circ \varphi^{-1}, \rho = \tau_t \circ \varphi^{-1}) \in C\left([0,T), T_{(\id,0)}G^m \right) \cap C^1\left([0,T), T_{(\id,0)}G^{m-1} \right)
	\]
	is a solution of the magnetic two-component Hunter-Saxton system:
	\begin{align}\label{(M2HS)}
		\begin{cases}
			u_{tx} = & -\frac{1}{2} u_x^2 - u \cdot u_{xx} + \frac{1}{2} \rho^2 - (s\rho + 2(c^2 - s\delta)), \\
			\rho_t = & -(\rho u)_x + s u_x,
		\end{cases} \tag{M2HS}
	\end{align}
	where \(c^2\), the \(\dot{H}^1\)-energy of the system, is a conserved quantity (integral of motion) given by
	\[
	c^2 = \GH_{(\varphi, \tau)}\left((\varphi_t, \tau_t), (\varphi_t, \tau_t) \right) = \langle (u, \rho), (u, \rho) \rangle_{\dot{H}^1} = \frac{1}{4} \int_{\mathbb{S}^1} u_x^2 + \rho^2 \, \mathrm{d}x.
	\]
	Additionally, the contact angle \(\delta\) is another conserved quantity, given by
	\[
	\delta = \left\langle \frac{\mathrm{d}}{\mathrm{d}t} \varPhi(\varphi, \tau), \mathrm{i} \cdot \varPhi(\varphi, \tau) \right\rangle_{L^2} = \frac{1}{2} \int_{\SS^1} \rho \, \mathrm{d}x.
	\]
	
\end{thm}

\begin{rems}
	\item \textbf{Novelty:} The \eqref{(M2HS)} is explicitly described as an infinite-dimensional magnetic geodesic equation.
	
	\item \textbf{Special case (\(s = 0\)):} When \(s = 0\), the system reduces to the two-component Hunter-Saxton system (2HS). As shown in \cite{Lennels13}, the (2HS) corresponds precisely to the geodesic equation on \((G, \GH)\).
	
	\item \textbf{Geometric insight:} The (M2HS) represents a linear deformation in \((u_x, \rho)\) of the (2HS). This deformation reflects the geometry of the equations, as the magnetic geodesic equation is itself a linear deformation of the geodesic equation in the velocity field; see \eqref{e:mg}.
	
	\item \textbf{Special case (\(\rho \equiv s\)):} By choosing \(\rho \equiv s\) and differentiating the resulting partial differential equation once with respect to \(x\), we recover the Hunter-Saxton equation; see \cite{km02, Lennels13}.
	
	\item \textbf{Maximal existence time:} The maximal existence time is independent of the Sobolev order \(H^m\), following exactly the lines of \cite[§4]{Lennels13}.
	
	\item \textbf{Tangent space at identity:} For the convenience of the reader, we recall that \(T_{(\id, 0)}G^m \cong H^m_0\left(\SS^1, \RR\right) \times H^{m-1}\left(\SS^1, \RR\right)\).

    \item \textbf{\eqref{(M2HS)} as a magnetic Euler–Arnold equation.} In \cite{Maier25magneticEulerArnold}, the author introduces a broader framework known as the \emph{magnetic Euler–
    	Arnold equation}. The system \eqref{(M2HS)} admits a formulation within this framework; however, its derivation—carried out in this chapter—is essential to rigorously establish that \eqref{(M2HS)} indeed fits into the magnetic Euler–Arnold setting.
\end{rems}

The proof of \Cref{t: magnetic Hunter Saxton system} consists of several steps. First, we need to explicitly express the pullback \(\varPhi^* \mathrm{d}\alpha\) to use \eqref{e:Lorentz} to compute the Lorentz force of the magnetic system \(\left(G, \GH, \varPhi^{*} \mathrm{d}\alpha\right)\). Along the way, we prove that the magnetic field \(\varPhi^* \mathrm{d}\alpha\) is right invariant under the group action of \(G\). This allows us to pull back the magnetic geodesic equation of \(\left(G, \GH, \varPhi^{*} \mathrm{d}\alpha\right)\) to the identity, which then leads to \Cref{t: magnetic Hunter Saxton system}.  
\subsection{The Lorentz force}\label{ss: 5.2 the lorentz force}We begin with the following lemma: 
\begin{lem}
	\label{l: Pullback of dalpha w.r.t. Phi is right invariant}
	The magnetic field \(\varPhi^* \mathrm{d}\alpha \in \Omega^2(G^m)\) is given at \((\varphi, \alpha) \in G^m\) by 
	\[
	\left( \varPhi^* \mathrm{d}\alpha \right)_{(\varphi, \alpha)}\left( (U_1,U_2), (V_1,V_2)\right) = \Im \int_{\mathbb{S}^1}\frac{1}{4 |\varphi_x|}\left(V_{1x}U_2 - U_{1x}V_2\right)  \varphi_x  \, \mathrm{d}x
	\] 
	for all \((U_1,U_2), (V_1,V_2) \in T_{(\varphi, \alpha)}G^m\). Furthermore, \(\varPhi^* \mathrm{d}\alpha\) is right invariant with respect to the right action of \(G^m\), i.e., 
	\[
	\left( \varPhi^* \mathrm{d}\alpha \right)_{(\varphi, \alpha)}\left( (U_1,U_2), (V_1,V_2)\right) = \left( \varPhi^* \mathrm{d}\alpha \right)_{(\mathrm{id}, 0)}\left( (U_1,U_2)\circ \varphi^{-1}, (V_1,V_2)\circ \varphi^{-1}\right)
	\] 
	holds for all \((\varphi, \alpha) \in G^m\) and all \((U_1,U_2), (V_1,V_2) \in T_{(\varphi, \alpha)}G^m\).
\end{lem} 

\begin{proof}
	Recall that the pullback of \(\mathrm{d}\alpha\) is formally defined as 
	\begin{equation}\label{e: pullback}
		\varPhi^{*}_{(\cdot)}\mathrm{d}\alpha(\cdot, \cdot) = \mathrm{d}\alpha_{\varPhi(\cdot)}(D\varPhi(\cdot), D\varPhi(\cdot)),
	\end{equation}
	where the derivative of the Madelung transform is, by \cite[Thm. 3.1]{Lennels13}, at the point \((\varphi, \tau) \in G^m\) given by 
	\begin{equation}\label{e: derivative of Madelung transform}
		\left(D\varPhi\right)_{(\varphi,\alpha)}\left( U_1,U_2\right)= \frac{1}{2\sqrt{\varphi_x}} \left( U_{1x}+ \mathrm{i} U_2 \varphi_x\right) \quad\forall (U_1,U_2)\in T_{(\varphi, \alpha)}G^m.
	\end{equation}
	Recall from the paragraph following \Cref{def:contact_form_infinite_dimension} that the derivative of the standard contact form is given as the restriction of \(\Im\langle \cdot, \cdot \rangle_{L^2}\) to \(T\SiL\). This, in combination with \eqref{e: pullback} and \eqref{e: derivative of Madelung transform}, gives the expression of \(\varPhi^* \mathrm{d}\alpha\) we want to prove. 
	
	In addition, a computation that follows exactly the lines of the fact that \(\GH\), see \cite[§2]{Lennels13}, is a right-invariant Riemannian metric on \(G^m\) proves that \(\varPhi^* \mathrm{d}\alpha\) is right invariant. 
\end{proof}
The Lorentz force of the magnetic system \(\left(G^m, \GH, \vardalp\right)\) is uniquely determined by \(\GH\) and \(\vardalp\). From \eqref{e: definition H1 dot metric at arbitrary element} and \Cref{l: Pullback of dalpha w.r.t. Phi is right invariant}, it can be directly derived that the Lorentz force of the magnetic system \(\left(G^m, \GH, \vardalp\right)\), analogous to \eqref{e:Lorenz force on SIl}, is given by
\begin{equation}\label{e: defining equation of Lorenz force}
	Y^{\dot{H^1}}{(\varphi,\tau)}\left( U_1,U_2 \right) = \left( J_{\dot{H}^1}\right)_{(\varphi, \tau)} \circ \left(\pi_{\xi_{\dot{H^1}}}\right)_{(\varphi, \tau)}(U_1,U_2),
\end{equation}
for all \(\left((\varphi,\tau),(U_1,U_2)\right)\in TG^m\).

Here, \(\JHo\) is an almost complex structure defined on the infinite-dimensional contact-type distribution \(\XiHo = \ker \varPhi^*\alpha\). Additionally, \(\pi_{\xi_{\dot{H^1}}}\) denotes the \(\GH\)-orthogonal projection of \(TG^M\) onto \(\XiHo\), which will be defined and computed in full detail later.

Following the computation in \Cref{l: Pullback of dalpha w.r.t. Phi is right invariant}, the pullback of the contact form \(\alpha\) with respect to \(\varPhi\) on \(G^m\) is given by
\begin{equation}\label{e: contact form on G}
	\left( \varPhi^*\alpha \right)_{(\varphi, \sigma)}\left(U_1, U_2\right) = -\frac{1}{2} \int_{\SS^1} U_2 \varphi_x \, \mathrm{d}x, \quad \forall \left((\varphi,\tau),(U_1,U_2)\right) \in TG^m.
\end{equation}
Thus, by \eqref{e: contact form on G}, the kernel of \(\varPhi^*\alpha\) at the point \((\varphi, \tau)\), called
 the contact distribution \(\XiHo\), is given by
\begin{equation}\label{e: definition Xi in H one}
	\left( \xi_{\dot{H}^1}\right)_{(\varphi,\tau)} := \ker \left( \varPhi^*\alpha \right)_{(\varphi, \tau)} = \left\{ (U_1,U_2)\in T_{(\varphi, \tau)}G^m : \int_{\SS^1} U_2 \varphi_x\, \mathrm{d}x = 0 \right\}.
\end{equation}
Thus, by \eqref{e: definition Xi in H one}, the \(\GH\)-orthogonal projection of \(TG^M\) onto \(\XiHo\) at the point \((\varphi, \tau)\in G^m\), where the dependence on the point is omitted for simplicity, is expressed as:
\begin{equation}\label{e: orthogonal projection onto Xi}
	\left(\pi_{\xi_{\dot{H^1}}}\right)_{(\varphi, \tau)} : T_{(\varphi, \tau)}G^m \to \left(\xi_{\dot{H^1}}\right)_{(\varphi, \tau)}, \quad (U_1, U_2) \mapsto \left( U_1, U_2 - \int_{\SS^1} U_2 \varphi_x \, \mathrm{d}x \right).
\end{equation}

\begin{rem}
	At this point, it is important to note that, since the Madelung transform is a magnetomorphism by \Cref{t: Madelungtransform as Magnetomorphims}, and since the contact distribution \(\xi\) on \(\SiL\) has codimension one in \(T\SiL\) (as shown in \Cref{p: contact distri on SIL has codim one}), the distribution \(\XiHo\), defined in \eqref{e: definition Xi in H one}, also has codimension one in \(TG^m\).
\end{rem}

Furthermore, the map \(J_{\dot{H}^1}: \XiHo \to \XiHo\) is a well-defined bundle endomorphism, which at \((\varphi,\tau)\in G^m\) is defined as:
\begin{equation} \label{d: acs on XiHo}
	\left( J_{\dot{H}^1}\right)_{(\varphi, \tau)}: \left( \xi_{\dot{H}^1}\right)_{(\varphi,\tau)} \longrightarrow  \left( \xi_{\dot{H}^1}\right)_{(\varphi,\tau)}, \quad 
	(U_1,U_2) \mapsto \left( \left(y\mapsto-\int_0^y U_2\varphi_x \, \mathrm{d}x\right), \frac{U_{1x}}{\varphi_x}\right).
\end{equation}

It is another exercise in logic that \( J_{\dot{H}^1} \) is the unique almost complex structure on \( \XiHo \) that is compatible with \( \varPhi^*\d\alpha \) and \( \GH \), i.e.,
\[
\left(J_{\dot{H}^1}\right)^2 = -\id_{\XiHo} \quad \text{and} \quad \restr{\left(\varPhi^*\d\alpha\right)}{\XiHo}(\cdot, J_{\dot{H}^1}\cdot) = \restr{\GH}{\XiHo}(\cdot, \cdot),
\]
where the details of the computation are left to the reader. By using \eqref{e: defining equation of Lorenz force}, \eqref{e: orthogonal projection onto Xi}, \Cref{d: acs on XiHo}, and the fact that \( Y^{\dot{H}^1} \) is right-invariant (since both \(\GH\) and \(\vardalp\) are), we derive the following lemma:

\begin{lem}\label{l: Lorenz force G with GH metric}
	The Lorentz force \( Y^{\dot{H}^1} \) of the magnetic system \(\left(G^m, \GH, \vardalp\right)\) at the point \((\varphi, \tau)\) is given by
	\[
	Y^{\dot{H}^1}_{(\varphi, \tau)}: T_{(\varphi, \tau)}G^m \to T_{(\varphi, \tau)}G^m,\quad 
	(U_1, U_2) \mapsto \left(\left(y \mapsto -\int_0^y \left(U_2 - \int_{\SS^1} U_2 \varphi_x \, \mathrm{d}x\right)\varphi_x \, \mathrm{d}x\right), \frac{U_{1x}}{\varphi_x}\right).
	\]
	Furthermore, \( Y^{\dot{H}^1} \) is right-invariant with respect to the right action of \( G^m \), i.e.,
	\[
	Y^{\dot{H}^1}_{(\id, 0)}((u, \rho)) = Y^{\dot{H}^1}_{(\varphi, \tau)}((u \circ \varphi, \rho \circ \varphi)) 
	\]
	for all \((\varphi, \tau) \in G^m\) and all \((u, \rho) \in T_{(\id, 0)}G^m\).
\end{lem}
\subsection{Derivation of the magnetic geodesic equation}\label{ss: 5.3 derivation of mag geode eq} The left-hand side of the magnetic geodesic equation \eqref{e:mg} is the Levi-Civita connection evaluated along a curve. For the magnetic system \(\left(G^m, \GH, \vardalp \right)\), this is given by \cite[Cor 3.3]{Lennels13} as:
\begin{equation}\label{e: Levi Cevita connection evaluated at identity}
	\left(\nabla_{(u, \rho)} (u,\rho )\right)_{(\id, 0)} = 
	\begin{pmatrix}
		u_t + u u_x\\
		\rho_t + u \rho_x
	\end{pmatrix}
	- 
	\begin{pmatrix}
		-\frac{1}{2}A^{-1}\partial_x\left(u_x^2 + \rho^2 \right)\\
		-\rho u_x
	\end{pmatrix},
\end{equation}
where \(A^{-1}\) is the inverse of the so-called inertia operator \(A\) (see \cite{AK98, Vi08} for general language), given by:
\begin{equation}\label{Inverse of Inertia operator}
	A^{-1}f(x) = - \int_0^x \int_0^y f(z)\, \d z\, \d y 
	+ x \int_{\SS^1}\int_0^y f(z)\, \d z\, \d y, \quad x \in \SS^1.
\end{equation}

By the right invariance of \(\GH\), the Levi-Civita connection of \((G^m, \GH)\) is also right invariant. Thus, using \eqref{e: Levi Cevita connection evaluated at identity}, the covariant derivative along \((\varphi, \tau) \in G^m\) applied to \((\varphi_t, \tau_t) \in T_{(\varphi, \tau)}G^m\) reads as:
\begin{equation*}
	\nabla_{(\varphi_t, \tau_t)}(\varphi_t, \tau_t) =
	\begin{pmatrix}
		\left(u_t + u u_x \right) \circ \varphi 
		- \frac{1}{2} \left( \int_0^{\varphi(\cdot)} \big(u_x^2(y) + \rho^2(y)\big)\, \d y 
		- \varphi(\cdot) \int_{\SS^1} \big(u_x^2(y) + \rho^2(y)\big)\, \d y \right)\\
		\left(\rho_t + u \rho_x + u_x \rho\right) \circ \varphi 
	\end{pmatrix}.
\end{equation*}

Using \Cref{l: Lorenz force G with GH metric}, and the fact that, by \Cref{t: Madelungtransform as Magnetomorphims}, the Madelung transform maps conserved quantities to conserved quantities as a magnetomorphism, we obtain the following result:

\begin{prop}\label{Prop: magnetic geodesic equation in G}
	The magnetic geodesic equation of \(\left(G^m, \GH, \vardalp \right)\) at the point \((\varphi, \tau) \in G^m\) is given by:
	\begin{align*}
		\begin{pmatrix}
			\left(u_t + u u_x \right)\circ \varphi - \frac{1}{2} \left( \int_0^{\varphi(\cdot)} \big(u_x^2(y) + \rho^2(y)\big) \d y - \varphi(\cdot)(c^2 - s \delta) \right) - s \int_{0}^{\varphi(\cdot)} \rho(y) \d y\\
			\left(\rho_t + (u\rho)_x - s u_x \rho \right)\circ \varphi
		\end{pmatrix}
		= 0,
	\end{align*}
	where \(c^2\), the \(\dot{H}^1\)-energy of the system, is a conserved quantity (integral of motion) given by:
	\[
	c^2 = \GH_{(\varphi, \tau)}\left((\varphi_t, \tau_t), (\varphi_t, \tau_t) \right) = \frac{1}{4} \int_{\mathbb{S}^1} \big(u_x^2 + \rho^2\big) \, \mathrm{d}x.
	\]
	Additionally, the contact angle \(\delta = \cos(\psi)\) is another conserved quantity, given by:
	\[
	\delta = \frac{1}{2} \int_{\SS^1} \rho \, \mathrm{d}x.
	\]
\end{prop}

\subsection{Proof of \Cref{t: magnetic Hunter Saxton system}}\label{ss: 5.4 proof of duality}
If we evaluate $Y^{\dot{H^1}}$ in \Cref{l: Lorenz force G with GH metric} at \((\id, 0)\) and insert it with \eqref{e: Levi Cevita connection evaluated at identity} into \eqref{e:mg}, we obtain that if \((\varphi, \tau)\) is a magnetic geodesic in \(\left(G^m, \GH, \vardalp \right)\), then \((u, \rho) := \left(\varphi_t \circ \varphi^{-1}, \tau_t \circ \varphi^{-1}\right)\) solves:
\begin{equation}\label{e: magentic geod equation of Gm at identity}
	\begin{pmatrix}
		u_t + u u_x\\
		\rho_t + u \rho_x
	\end{pmatrix} 
	- 
	\begin{pmatrix}
		-\frac{1}{2}A^{-1}\partial_x\left(u_x^2 + \rho^2 \right)\\
		-\rho u_x
	\end{pmatrix}
	= s \begin{pmatrix}
		\left(y \mapsto -\int_0^y \left( \rho - \int_{\SS^1}\rho \, \d x\right)\d x\right)\\
		u_x
	\end{pmatrix}.
\end{equation}

By differentiating the first row of \eqref{e: magentic geod equation of Gm at identity} with respect to \(x\), we get:
\begin{equation}\label{e: differentiating first row of magnetic geodesic at identity}
	u_{tx} + u_x^2 + u u_{xx} + \partial_x \frac{1}{2}A^{-1}\partial_x\left(u_x^2 + \rho^2 \right) = -s\left(\rho - \int_{\SS^1} \rho \, \d x\right),
\end{equation}
where one can show with a standard computation (left to the reader) that:
\begin{align*}
	\frac{1}{2}\partial_x A^{-1}\partial_x\left(u_x^2 + \rho^2 \right) 
	= -\frac{1}{2} \left(u_x^2(x) + \rho^2(x)\right) + \frac{1}{2}\int_{\SS^1}\left(u_x^2 + \rho^2\right)\, \d x.
\end{align*}

Using \Cref{Prop: magnetic geodesic equation in G} and \eqref{e: differentiating first row of magnetic geodesic at identity}, we find that:
\[
u_{tx} + \frac{1}{2}u_x^2 + u u_{xx} - \frac{1}{2}\rho^2 = -(s\rho + 2(c^2 - \delta s)),
\]
where \(c^2\) is the conserved \(\dot{H}^1\)-energy and \(\delta\) is the conserved contact angle.

Substituting this back into \eqref{e: magentic geod equation of Gm at identity}, we end up with the magnetic two-component Hunter saxton system:
\[
\begin{cases}
	u_{tx} &= -\frac{1}{2}u_x^2 - u \cdot u_{xx} + \frac{1}{2}\rho^2 - (s\rho + 2(c^2 - \delta s)),\\
	\rho_t &= -(u\rho)_x + s u_x.
\end{cases}
\]

This concludes the proof of \Cref{t: magnetic Hunter Saxton system}.
\subsection{Infinitely many conserved quantities of the (M2HS)}\label{ss:infinitely many conserved quantities}

As a first illustration of the Hamiltonian, respectively geometric, nature of \eqref{(M2HS)}, we aim to prove, by means of the Madelung transform, that the (M2HS) possesses infinitely many conserved quantities.

\begin{cor}\label{c: infinitely many conserved quantities}
	The magnetic system \(\left(G^m, \GH,  \vardalp\right)\) possesses infinitely many conserved quantities, and thus so does \eqref{(M2HS)}.
\end{cor}

\begin{proof}
	The fact that the magnetic system \(\left(G^m, \GH,  \vardalp\right)\) possesses infinitely many conserved quantities follows directly from \Cref{c: magnetic system SiL poss infty many integrals} in combination with \Cref{t: Madelungtransform as Magnetomorphims}. Thus, by \Cref{t: magnetic Hunter Saxton system}, the system \eqref{(M2HS)} also possesses infinitely many conserved quantities.
\end{proof}

\section{Geometry of blow-ups and global weak solutions}
In this section, we exploit the Hamiltonian structure of the system, in the sense of infinite-dimensional geometry, to obtain analytical insights into the singularity formation of \eqref{(M2HS)}. From the Hamiltonian perspective, there is a natural way to extend solutions of (M2HS) beyond blow-up in a weak sense. This approach ultimately leads to global weak solutions of \eqref{(M2HS)}.

We begin with an overview of the section. In \Cref{ss: 6.1 geoemtry of blow ups}, we establish precise blow-up criteria for \eqref{(M2HS)}. These criteria allow one to determine from the initial data whether a solution develops singularities; see \Cref{C: blow up M2HS with special initial values}. The analysis is based on the geometric interpretation of blow-up phenomena given in \Cref{t: geometric interpret of blow ups}.

In \Cref{ss: 6.2 global weak solutions}, we show how, after relaxing the ambient space, one can use geometric arguments to obtain global existence of a weak flow; see \Cref{d: weak magnetic geodesics in MAC} for the precise definition and \Cref{t: global existence weak geodesic flow}. By a duality argument similar to \Cref{t: magnetic Hunter Saxton system}, this yields the existence of global weak solutions of \eqref{(M2HS)} in the sense of \Cref{d: weak solutions}; see \Cref{t: global weak solutions }. 

Throughout this section, the arguments are primarily geometric; the more involved analytical details are deferred to \cite[§4]{Maier26Globalweak}. For the reader’s convenience, the corresponding statements are recalled at the beginning of the relevant discussions.
\subsection{Geometric criteria for blow-ups}\label{ss: 6.1 geoemtry of blow ups} First, we provide the \( C^0 \)-blowup of the first spatial derivative of \( u \) with a geometric interpretation by extending similar results on the Hunter-Saxton equation and the two-component Hunter-Saxton system in \cite{l07.1, Lennels13} to our setting. 

\begin{thm}\label{t: geometric interpret of blow ups}
	Let \((u, \rho)\) be a solution of \Cref{(M2HS)}, satisfying the prerequisites of \Cref{t: magnetic Hunter Saxton system}. Then there exists a \( T > 0 \) such that
	\[
	\lim_{t \nearrow T} \| u_x(t, \cdot) \|_{C^0} = \infty
	\]
	if and only if, for \((\varphi, \tau) \in G^m\), defined as in \Cref{t: magnetic Hunter Saxton system}, the following holds:
	\[
	\lim_{t \nearrow T} \varPhi\left( \varphi(t), \tau(t)\right) \in \partial \U^{m-1}.
	\]
\end{thm}

\begin{rem}[Geometric insight]
	\item The geometric idea behind \Cref{t: magnetic Hunter Saxton system} is intuitive, as we now illustrate. The green curve in \Cref{f: geometric blow up criteria} represents the image \( \varPhi(\varphi, \tau) \) of the Madelung transform of the magnetic geodesic \((\varphi, \tau)\). The curve $\varPhi(\varphi, \tau)$ lies in \( \U^{m-1} \), and by \Cref{t: Madelungtransform as Magnetomorphims}, is a magnetic geodesic in \( (\U^{m-1}, \GL, \d\alpha )\). In \Cref{f: geometric blow up criteria}, the green curve, $\varPhi(\varphi, \tau)$, intersects the boundary \( \partial \U^{m-1} \) if and only if the corresponding solution \((u, \rho)\) of \Cref{(M2HS)} exhibits a blowup in the sense of \Cref{t: magnetic Hunter Saxton system}. This geometric criterion provides a clear and intuitive way to connect the blowup behavior of \((u, \rho)\) with the behavior of the corresponding magnetic geodesics $\varPhi(\varphi, \tau)$ in \( (\U^{m-1}, \GL, \d\alpha )\).
\end{rem}
\begin{figure}
	\centering
	\includegraphics[width=0.5\textwidth]{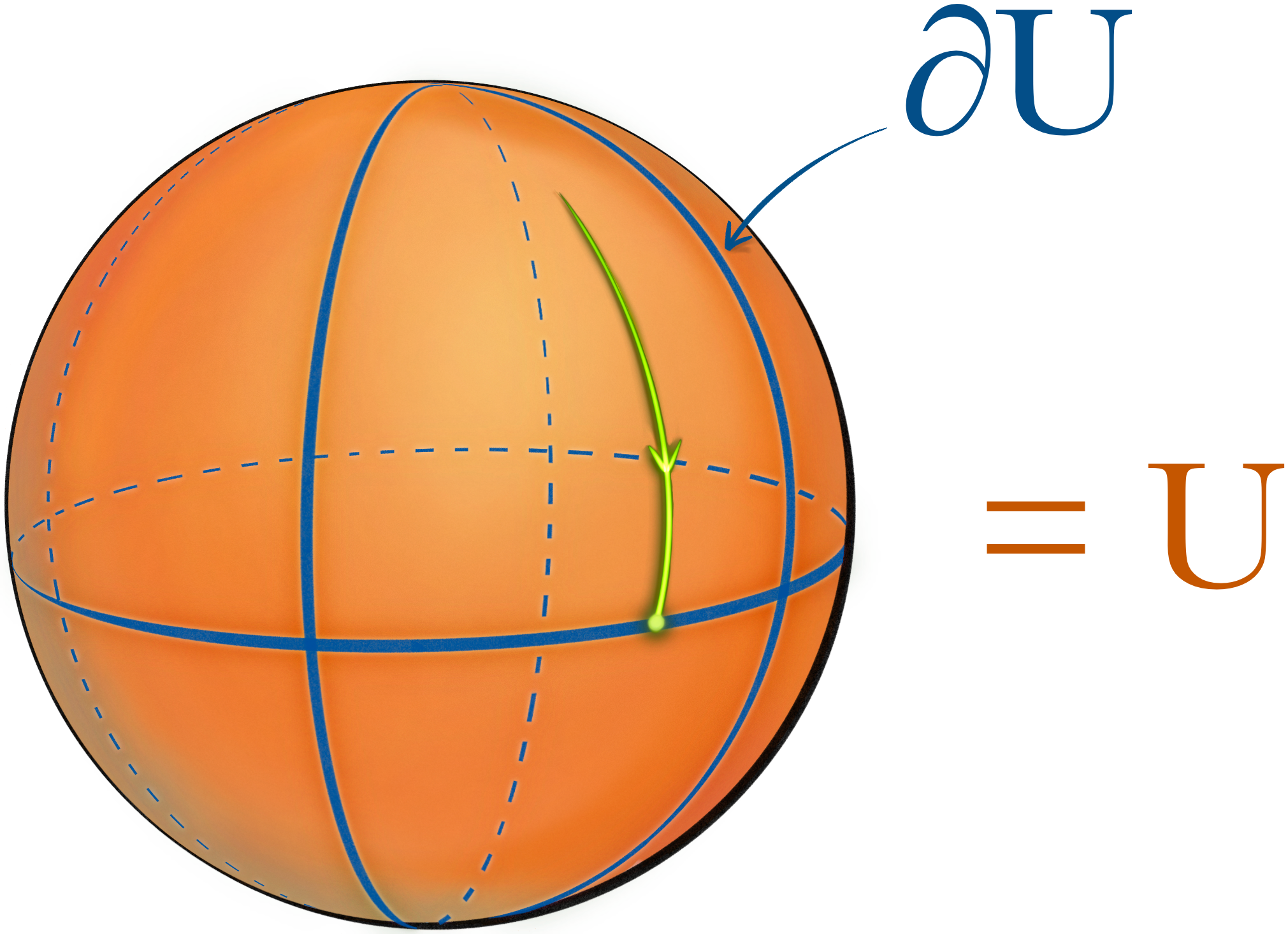}
	\caption{This is an illustration of \Cref{t: geometric interpret of blow ups}. The main idea is that \ref{(M2HS)} exhibits a blowup if and only if the corresponding magnetic geodesic \( \varPhi(\varphi, \tau) \) in $U$, shown as the green curve, intersects the boundary \( \partial U \). (Illustration by Ana Chavez Caliz.)}
	\label{f: geometric blow up criteria}
\end{figure}
\begin{proof}
Let $(\varphi, \tau)$ be as defined in \Cref{t: geometric interpret of blow ups}. By the definition of the Madelung transform (see \eqref{e: Madelung transform is isometry}) and the boundary condition on $\partial \U^{m-1}$, there exists a $T > 0$ and an $x_0 \in \SS^1$ such that 
\[
\lim_{t \nearrow T} \sqrt{\varphi_x(t,x_0)} e^{\i \tau(t,x_0)/2} = 0.
\]
Since the exponential map is nowhere vanishing, it follows that $\lim_{t \nearrow T} \sqrt{\varphi_x(t,x_0)} = 0$,
and hence  $\lim_{t \nearrow T} \varphi_x(t,x_0) = 0$. From the relation $\partial_t \ln (\varphi_x) = u_x \circ \varphi$, and applying the Fundamental Theorem of Calculus, we deduce that 
\[
\lim_{t \nearrow T} \vert (u \circ \varphi)(t,x_0) \vert = \infty.
\]
Since each step in the reasoning is reversible, this completes the proof.
\end{proof}

So in other words, \Cref{t: geometric interpret of blow ups} allows us to reformulate the problem of a possible blow-up of the \eqref{(M2HS)} as a problem in terms of magnetic geodesics in $(\SiL, \GL, \d\alpha)$. Due to the dynamical reduction theorem \Cref{cor: dynamical reduction of SiL to Sd}, this can be tackled with methods from \cite{ABM23}.
As a result of this, we can determine from the initial data of a solution of \eqref{(M2HS)} whether this solution has a blow-up:
\begin{cor}\label{C: blow up M2HS with special initial values}
	Let \((u, \rho)\) be a solution of \Cref{(M2HS)}, satisfying the prerequisites of \Cref{t: magnetic Hunter Saxton system}, with initial data $(u_0, \rho_0) \in T_{(\id,0)}G^m$ so that $\GH_{(\id,0)}((u_0,\rho_0), (u_0, \rho_0))=1$. Then $(u, \rho)$ has a blow-up in the sense of \Cref{t: geometric interpret of blow ups} if and only if there exists an $x_0$ such that $\rho_0(x_0) = s$. 
\end{cor}

\begin{rems}		
	\item \textbf{Special case \(s = 0\):} By choosing \(s = 0\), we recover a similar statement in the case of the (2HS) from \cite[Cor. 3.5]{Lennels13}.
	\item \textbf{Solution formula:} This can be independently proven using a general solution formula for \eqref{(M2HS)}, which was developed by the author in~\cite[§3]{Maier26Globalweak}.
\end{rems}

Before proving \Cref{C: blow up M2HS with special initial values}, it is useful to reformulate the problem by describing magnetic geodesics of \((G^{m}, \GH, \vardalp)\) with energy \(k\) as unit speed magnetic geodesics for the family of systems \((G, \GH, s\vardalp)\), where \(s = \frac{1}{\sqrt{2k}}\) represents the strength of the magnetic geodesic. Moreover, by \Cref{t: Madelungtransform as Magnetomorphims} and \Cref{t: geometric interpret of blow ups}, the proof of \Cref{C: blow up M2HS with special initial values} is reduced to proving the following lemma.

\begin{lem}
	\label{l: magnetic geodesic touches bound of U iff rho(x)=s}
	Let $\gc$ be a unit speed magnetic geodesic in $(\U^{m-1}, \GL, s\d\alpha)$ with initial values \(\gc(0) = \varPhi(\id,0)\) and \(\dgc(0)= \D\varPhi(u_0,\rho_0)\), with $ (u_0,\rho_0) \in T_{(\id,0)}G^m$. Then there exists a \(T > 0\) such that 
	\[
	\lim_{t \nearrow T}\gamma(t) \in \partial\U^{m-1}
	\]
	if and only if there exists an \(x_0 \in \SS^1\) such that \(\rho(x_0) = s\).
\end{lem}

\begin{proof}
	Let $\gc$ be as described in \Cref{l: magnetic geodesic touches bound of U iff rho(x)=s}. Then, $\gc$ is also a magnetic geodesic in $\left(\SiL, \GL, \d\alpha \right)$. Consequently, by \Cref{cor: dynamical reduction of SiL to Sd}, it is also a magnetic geodesic in $(\Sd, g, \d\alpha)$. According to \cite[Thm. 1.9, Eq. 3.3, Rmk. 3.1]{ABM23}, the geodesic $\gc(t)$ takes the explicit form:
	\begin{equation}\label{e: explicit form of magnetic geodesics on S3}
		\gc(t) = \frac{1}{\theta_2-\theta_1}\left( \theta_2 \gc(0) + \i \dgc(0) \right)e^{\i\theta_1 t} 
		+ \frac{1}{\theta_1-\theta_2}\left( \theta_1 \gc(0) + \i \dgc(0) \right)e^{\i\theta_2 t},
	\end{equation}
	where the parameters $\theta_1$ and $\theta_2$ are defined as 
	\begin{equation}\label{e: definition theta1/theta 2}
		\theta_{1/2} = \frac{s}{2} \pm \vert C_s(\psi) \vert = \frac{s \pm \sqrt{s^2 - 4(1 - s\delta)}}{2}.
	\end{equation}
	By explicitly using the initial values $\gc(0)$ and $\dgc(0)$ as provided in \Cref{l: magnetic geodesic touches bound of U iff rho(x)=s}, \eqref{e: derivative of Madelung transform}, and \eqref{e: explicit form of magnetic geodesics on S3}, it follows that there exists a time $T > 0$ such that $\gc(T) \in \partial\U^{m-1}$ if and only if:
	\begin{equation}\label{e: gamma in parttial U iff and only if = 0 with init values}
		0 = \frac{1}{\theta_2-\theta_1}\left( \theta_2 + \frac{\i}{2} \left(u_x(x_0) + \i \rho(x_0)\right) \right)e^{\i\theta_1 T} 
		+ \frac{1}{\theta_1-\theta_2}\left( \theta_1 + \frac{\i}{2} \left(u_x(x_0) + \i \rho(x_0)\right) \right)e^{\i\theta_2 T}.
	\end{equation}
	By applying Euler's formula to the complex exponential functions in \eqref{e: gamma in parttial U iff and only if = 0 with init values} and comparing the real and imaginary parts, we obtain that \eqref{e: gamma in parttial U iff and only if = 0 with init values} holds true if and only if:
	\begin{align}
		\theta_1+\frac{1}{2}\rho(x_0) &= \cos((\theta_1-\theta_2)T)(\theta_2+\rho(x_0)) - \sin((\theta_1-\theta_2)T)\frac{1}{2}u_x(x_0), \label{e: equation for rho(x0)} \\
		\frac{1}{2}u_x(x_0) &= \sin((\theta_1-\theta_2)T)\left(\theta_2+\frac{1}{2}\rho(x_0)\right)+\cos((\theta_1-\theta_2)T) \frac{1}{2}u_x(x_0). \label{e: equation for u_x(x_0)}
	\end{align}
	Now there are two cases: either $1-\cos((\theta_1-\theta_2)T) = 0$ or $1-\cos((\theta_1-\theta_2)T) \neq 0$. The first case cannot occur, since it implies $(\theta_1-\theta_2)T \in \frac{2\pi}{\theta_1-\theta_2}\ZZ$, so $\sin((\theta_1-\theta_2)T) = 0$. This allows us to rewrite \eqref{e: equation for rho(x0)} and \eqref{e: equation for u_x(x_0)} as:
	\begin{align*}
		\theta_1+\rho(x_0) &= \theta_2+\rho(x_0) \Longrightarrow 0 = \theta_1-\theta_2 = \sqrt{s^2+4(1-s\delta)},
	\end{align*}
	which implies, following the computations in \cite[§3]{ABM23}, that $\gc$ is a Reeb orbit in $(\SiL, \alpha)$. However, this is not the case as $\i \gc(0) \neq \dgc(0)$. 
	
	In the second case, we can solve \eqref{e: equation for u_x(x_0)} for $\frac{1}{2}u_x(x_0)$, substitute this expression into \eqref{e: equation for rho(x0)}, and obtain, using \eqref{e: definition theta1/theta 2}:
	\begin{equation*}
		\theta_1+\frac{1}{2}\rho(x_0) = -\theta_2-\frac{1}{2}\rho(x_0) \Longrightarrow \rho(x_0) = s.
	\end{equation*}
	The reverse direction of the proof follows exactly the same lines as the forward direction and is left as an exercise for the interested reader. This concludes the proof.
\end{proof}
\subsection{The Existence of Global Weak Solutions of (M2HS) through the Geometric Lens}\label{ss: 6.2 global weak solutions}

At this point, the reader might be wondering about the fact that the magnetic geodesic $\gc$ in \Cref{l: magnetic geodesic touches bound of U iff rho(x)=s}, visualized as the green curve in \Cref{f: geometric blow up criteria}, does not exhibit a blow-up. This raises the question: what happens to the corresponding solution of \eqref{(M2HS)} if we continue the magnetic geodesic—the green curve in \Cref{f: geometric blow up criteria}—beyond the boundary $\partial U^{m-1}$?

Before answering this question, we note that in this paper, we focus on the geometric constructions associated with this situation. The analytical details regarding weak solutions will be provided in \cite{Maier26Globalweak}. 

We begin by introducing the natural completion $\M_{H^1}$ of $G^m$ through the lens of geometry, which we intuitively think of as the inverse of the Madelung transform applied to $\SiL$. This is defined as in \cite{wu11}:
\begin{equation}\label{e: definition M_AC}
	\MAC := M_{H^1} \rtimes L^2\left( \SS^1, \RR \right),
\end{equation}
where the group multiplication is defined as on $G^m$, and $M_{H^1}$ is defined as:
\[
M_{H^1} := \{ f \in H^1\left([0,1], [0,1]\right) : f(0) = 0, \ f(1) = 1, \ f \ \text{is nondecreasing}\}.
\]

\begin{rem}\label{r: MAC as natural completion of Gm}
	The space $\MAC$ is the natural completion of $G^m$. As we have seen in the proof of \Cref{t: geometric interpret of blow ups}, a solution $(u,\rho)$ has a blow-up if and only if, for the corresponding magnetic geodesic $(\varphi,\tau)$ in $(G^m, \GH, \vardalp)$, there exists a $T>0$ and an $x_0 \in \SS^1$ such that $\lim_{t \nearrow T}\varphi_x(x_0,t)=0$. In other words, $\lim_{t \nearrow T}\varphi(\cdot, t)$ is no longer an element of $\Diff_0(\SS^1)$, with $\SS^1 = \RR/\ZZ$, but is still an element in $M_{H^1}$. 
\end{rem}

Thus, we can naturally extend the magnetic system $(G^m, \GH, \vardalp)$ to $\left(\MAC, \GH, \vardalp \right)$. With this extension in mind, we define the notion of a weak magnetic geodesic as follows:

\begin{defn}[{\cite[Def. 4.4]{Maier26Globalweak}}]\label{d: weak magnetic geodesics in MAC}
A curve \( (\varphi, \tau) \) with initial data \( (u_0, \rho_0) \in T_{(\id, 0)}\MAC \) is called a \emph{weak magnetic geodesic} of \( (\MAC, \GH, \varPhi^* \mathrm{d}\alpha) \) if the following conditions are satisfied:
\begin{enumerate}[label=(\arabic*)]
	\item \label{it:weak_geod:in_MAC}
	The curve \( (\varphi, \tau) \) remains in \( \MAC \) for all \( t \geq 0 \); that is,
	\[
	\bigl[x\mapsto(\varphi(t, x), \tau(t, x))\bigr] \in \MAC \quad \forall t \in [0, \infty) \, .
	\]
	
	\item \label{it:weak_geod:tangent}
	The time derivative \( (\varphi_t, \tau_t) \) lies in the tangent space \( T\MAC \) for all \( t \geq 0 \); that is,
	\[
	\bigl[x\mapsto(\varphi_t(t, x), \tau_t(t, x))\bigr] \in T_{(\varphi(t, \cdot), \tau(t, \cdot))} \MAC \quad \forall t \in [0, \infty)
	\, .
	\]
	
	\item \label{it:weak_geod:energy}
	The kinetic energy is constant almost everywhere in time; that is, for almost every \( t \in [0, \infty) \),
	\[
	\left\langle (\varphi_t(t), \tau_t(t)), (\varphi_t(t), \tau_t(t)) \right\rangle_{\dot{H}^1, (\varphi, \tau)} =
	\left\langle (\varphi_t(0), \tau_t(0)), (\varphi_t(0), \tau_t(0)) \right\rangle_{\dot{H}^1, (\id, 0)} \, .
	\]
	Moreover, the contact angle is conserved almost everywhere; that is, for almost every \( t \in [0, \infty) \),
	\[
	\delta := \int_{\SS^1} \rho_0(x)\, \mathrm{d}x = \int_{\SS^1} \tau_t(t, x)\, \mathrm{d}x \, .
	\]
	
	\item \label{it:weak_geod:equation}
	For almost every \( t \in [0, \infty) \), the curve \( (\varphi, \tau) \) satisfies the magnetic geodesic equation as given in \Cref{Prop: magnetic geodesic equation in G}, corresponding to the magnetic system \( (G^m, \GH, \varPhi^* \mathrm{d}\alpha) \).
\end{enumerate}
We further call \( (\varphi, \tau) \) a \emph{unit speed weak magnetic geodesic} if the initial data satisfies
\[
\left\langle (u_0, \rho_0), (u_0, \rho_0) \right\rangle_{\dot{H}^1} = 1 \, .
\]
\end{defn}

\begin{thm}[{\cite[Thm. 4.5]{Maier26Globalweak}}]\label{t: global existence weak geodesic flow}
Assume $(u_0,\rho_0)\in T_{(\id,0)}\MAC$ satisfies $\langle (u_0,\rho_0),(u_0,\rho_0)\rangle_{\dot H^1}=1$ and define $\theta_{1/2}$ by
\[
\theta_{1/2}=\frac{s\pm\sqrt{s^2+4(1-s\delta)}}{2}.
\]
For $(t,x)\in[0,\infty)\times\SS^1$ set
\begin{equation}\label{eq:phi_def_global}
	\varphi(t,x)
	=
	\int_0^x
	\left|
	\frac{1}{\theta_2-\theta_1}\Bigl[
	\theta_2 e^{\i \theta_1 t}-\theta_1 e^{\i\theta_2 t}
	+\frac{\i}{2}\bigl(u_{0,x}(y)+\i \rho_0(y)\bigr)\bigl(e^{\i\theta_1 t}-e^{\i\theta_2 t}\bigr)
	\Bigr]
	\right|^2
	\,\d y,
\end{equation}
and
\begin{equation}\label{eq:tau_def_global}
	\tau(t,x)
	=
	\rho_0(x)\int_0^t
	\chi_{\{\varphi_x(s,\cdot)>0\}}\,\frac{1}{\varphi_x(s,x)}\,\d s.
\end{equation}
Then $(\varphi,\tau)$ is a global unit speed weak magnetic geodesic on $(\MAC,\GH,\varPhi^*\d\alpha)$.
\end{thm}

\begin{figure}
	\centering
	\includegraphics[width=0.5\textwidth]{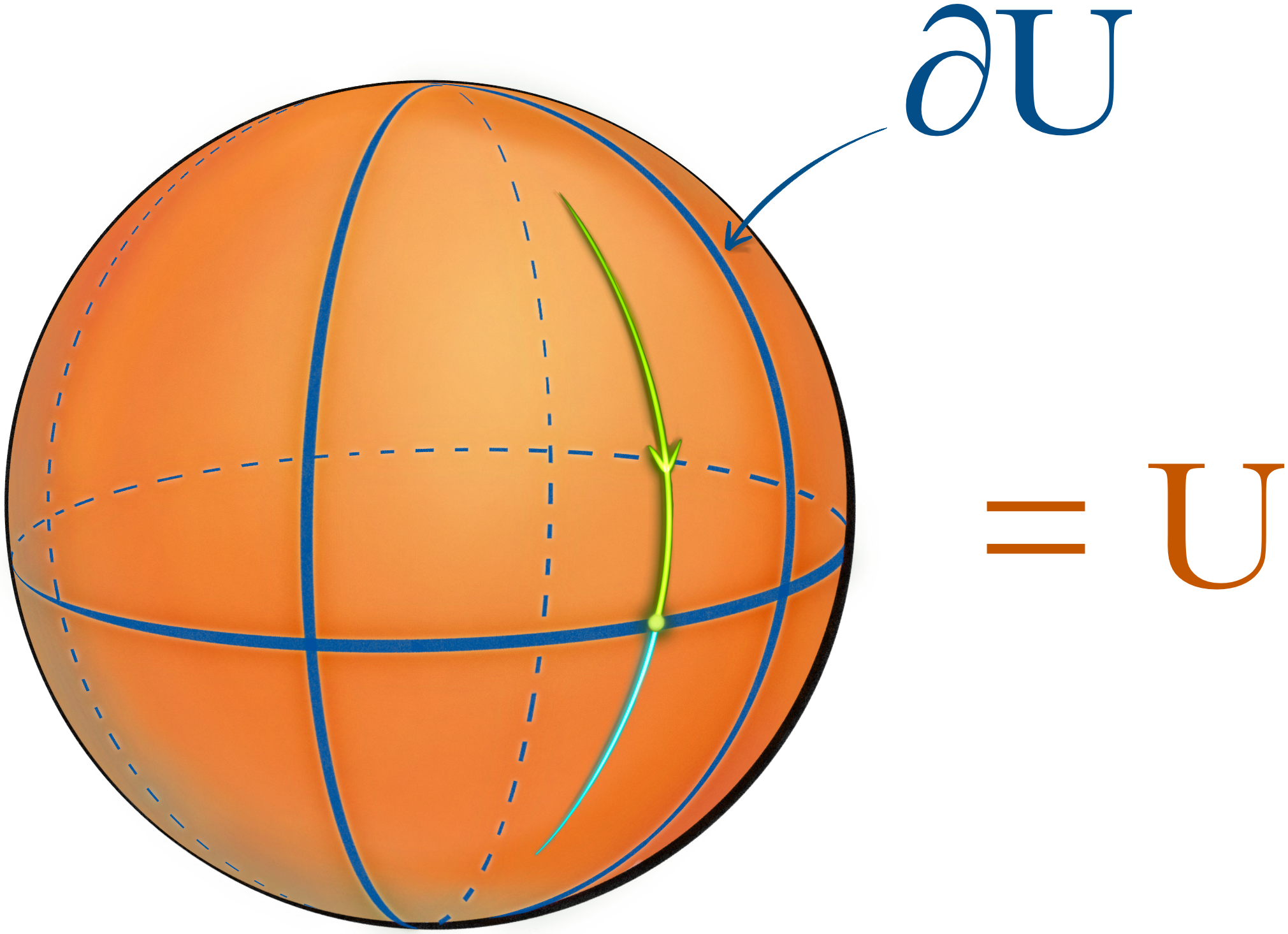}
	\caption{Illustration of the proof of \Cref{t: global existence weak geodesic flow}. The key idea is that, using \Cref{t: Madelungtransform as Magnetomorphims} and \Cref{t: geometric interpret of blow ups}, it suffices to extend the green curve $(\varPhi(\varphi), \varPhi(\tau))$ beyond the boundary. As $(\varPhi(\varphi), \varPhi(\tau))$ is, by \Cref{t: Madelungtransform as Magnetomorphims} and \cite[Prop. 6.4]{ABM23}, a magnetic geodesic in $(\U, \GL, \d\alpha)$, and therefore also in $(\SiL, \GL, \d\alpha)$, the corresponding magnetic geodesic, represented by the blue curve, exists globally by \Cref{cor: dynamical reduction of SiL to Sd}. This allows the extension of the green curve beyond $\partial \U$. (Illustration by Ana Chavez Caliz.)}
	\label{f: geometric construct global weak magnetic geodesic}
\end{figure}

\begin{proof}
	We present a geometrically inspired proof strategy, with the detailed analytical arguments outlined in \cite[§4.4]{Maier26Globalweak}. Let $(\varphi, \tau)$ be a magnetic geodesic in $(\MAC, \GH, \vardalp)$. By \Cref{t: Madelungtransform as Magnetomorphims} and \cite[Prop. 6.4]{ABM23}, the curve $(\varPhi(\varphi), \varPhi(\tau))$ is a magnetic geodesic in $(\U, \GL, \d\alpha)$, which we denote by $\gc$. By the definition of $\U$, the curve $\gc$ is also a magnetic geodesic in $(\SiL, \GL, \d\alpha)$. Therefore, by \Cref{cor: dynamical reduction of SiL to Sd}, the magnetic geodesic $\gc$ exists globally in $(\SiL, \GL, \d\alpha)$. 
	
	Using \Cref{t: geometric interpret of blow ups}, it suffices to extend $(\varPhi(\varphi), \varPhi(\tau))$ beyond the boundary $\partial \U$. This extension is achieved by following the magnetic geodesic $\gc$, which allows $(\varPhi(\varphi), \varPhi(\tau))$ to continue beyond $\partial \U$ within $\U$. As a result, $(\varPhi(\varphi), \varPhi(\tau))$ exists globally. By applying the inverse of the Madelung transform, we conclude that the corresponding magnetic geodesic in $(\MAC, \GH, \vardalp)$ also exists globally, as illustrated in \Cref{f: geometric construct global weak magnetic geodesic}.
\end{proof}
In the spirit of \Cref{t: magnetic Hunter Saxton system}, we aim to use \Cref{t: global existence weak geodesic flow} to establish global weak solutions of \eqref{(M2HS)}. To achieve this, we extend the previous work on the geodesic case in \cite{l07.1}, \cite{wu11}, introducing the notion of weak solutions for \eqref{(M2HS)} as follows:

\begin{defn}[{\cite[Def. 4.8.]{Maier26Globalweak}}]\label{d: weak solutions}
	A pair \( (u, \rho) \colon [0,\infty) \times \SS^1 \to \RR \) is called a \emph{global conservative weak solution} of~\eqref{(M2HS)} with initial data \( (u_0, \rho_0) \in H^1(\SS^1, \RR) \times L^2(\SS^1, \RR) \) if and only if the following hold:
	\begin{enumerate}[label=(\arabic*)]
		\item \label{it:weak_sol:H1}
		For all \( t \in [0, \infty) \), we have \( x\mapsto u(t, x) \in H^1(\SS^1, \RR) \).
		
		\item \label{it:weak_sol:cont_init}
		The map \( u \colon [0,\infty) \times \SS^1 \to \RR \) is continuous, and the initial conditions are satisfied pointwise and almost everywhere, respectively:
		\[
		u(0, x) = u_0(x) \quad \text{for all } x \in \SS^1, 
		\qquad 
		\rho(0, x) = \rho_0(x) \quad \text{for a.e. } x \in \SS^1.
		\]
		
		\item \label{it:weak_sol:Linfty}
		The maps \( t \mapsto u_x(t,\cdot) \) and \( t \mapsto \rho(t,\cdot) \) belong to \( L^\infty([0,\infty), L^2(\SS^1, \RR)) \).
		
		\item \label{it:weak_sol:equation}
		The map \( t \mapsto u(t,\cdot) \) is absolutely continuous from \( [0,\infty) \) into \( L^2(\SS^1, \RR) \), and for almost every \( t \in [0,\infty) \), the pair \( (u(t,\cdot), \rho(t,\cdot)) \) satisfies the equation in \Cref{Prop: magnetic geodesic equation in G} evaluated at \( (\id, 0) \); that is,
		\[
		\begin{pmatrix}
			u_t + u u_x \\
			\rho_t + u \rho_x
		\end{pmatrix}
		=
		\begin{pmatrix}
			-\frac{1}{2} A^{-1} \partial_x (u_x^2 + \rho^2) - s \cdot \int_0^x \left( \rho - \int_{\SS^1} \rho \, \mathrm{d}x \right) \mathrm{d}x \\
			-(\rho u)_x + s u_x
		\end{pmatrix}
		\quad \text{in } L^2(\SS^1, \RR) \times L^2(\SS^1, \RR) .
		\]
	\end{enumerate}
\end{defn}
Through the lens of Hamiltonian dynamics, the next theorem follows directly from \Cref{t: magnetic Hunter Saxton system} and \Cref{t: global weak solutions }. However, this requires rigorous analytical work, which has been carried out in \cite[§4.5]{Maier26Globalweak} by the author in a more general setting.

\begin{thm}[{\cite[Thm. 4.12.]{Maier26Globalweak}}]\label{t: global weak solutions }
	Let \( (\varphi, \tau) \) be a weak unit speed magnetic geodesic in \( \MAC \) with initial data \( (u_0, \rho_0) \in H^1(\SS^1, \RR) \times L^2(\SS^1, \RR) \).\\
	Then the pair defined by
	\begin{equation}\label{eq:weak_solution_def}
		\begin{pmatrix}
			u(t, \varphi(t,x)) \\
			\rho(t, \varphi(t,x))
		\end{pmatrix}
		:=
		\begin{pmatrix}
			\varphi_t(t,x) \\
			\tau_t(t,x)
		\end{pmatrix}, \quad \text{for } (t,x) \in [0,\infty) \times \SS^1,
	\end{equation}
	is a global conservative weak solution of~\eqref{(M2HS)} with initial data \( (u_0, \rho_0) \).
	
	Furthermore, the energy is an integral of motion almost everywhere in time; that is, for almost every \( t \in [0, \infty) \),
	\begin{equation}\label{eq:energy_integral_of_motion}
		\left\langle (u(t), \rho(t)), (u(t), \rho(t)) \right\rangle_{\dot{H}^1}
		=
		\left\langle (u_0, \rho_0), (u_0, \rho_0) \right\rangle_{\dot{H}^1}
		= 1\, .
	\end{equation}
	In addition, the contact angle is conserved almost everywhere; that is, for almost all \( t \in [0, \infty) \),
	\[
	\delta = \int_{\SS^1} \rho_0(x)\, \mathrm{d}x = \int_{\SS^1} \tau_t(t,x)\, \mathrm{d}x \, .
	\]
\end{thm}

\begin{rems}
	\item \textbf{Special case \(s=0\):} By choosing \(s=0\) in \Cref{Prop: magnetic geodesic equation in G}, we recover the existence of global weak solutions for the two-component Hunter-Saxton system. This corresponds to the main result presented in \cite{wu11}.
	
	\item \textbf{Special case \(\rho \equiv s\):} By setting \(\rho = s\) in \Cref{Prop: magnetic geodesic equation in G}, we recover the existence of global weak solutions for the Hunter-Saxton equation. This result is the main focus of \cite{l07.2}.
	
	\item \textbf{Fixing the kinetic energy:} Fixing the \(\dot{H}^1\)-energy of a magnetic geodesic in \((\MAC, \GH, s\cdot \vardalp)\) to 
	\[
	\GH_{(\id,0)}((u_0,\rho_0), (u_0,\rho_0))=1,
	\] 
	corresponds to considering unit-speed magnetic geodesics in \((\SiL, \GL, s\cdot \d\alpha)\). This equivalence arises due to the normalization factor in the exponent of \eqref{e: Madelung transform is isometry}.
\end{rems}

\section{Mañé's critical value and the Hopf--Rinow theorem for the (M2HS)}
The goal of this section is to extend \Cref{t:mane for SiL} to the magnetic system 
\((\MAC, \GH, \vardalp)\) and, based on this extension, to establish a Hopf--Rinow type theorem for solutions of \eqref{(M2HS)}. This constitutes one of the main results of the present work:

\begin{thm}\label{t: Manes critical value for (M2HS)}
	Mané's critical value for the system is given by:
	\[
	c(\MAC, \GH, \vardalp) = \tfrac{1}{2}\Vert\alpha\Vert_\infty^2 = \tfrac{1}{8}.
	\]
	Let \(q_0 := (\varphi_0, \tau_0)\) and \(q_1 := (\varphi_1, \tau_1)\) be two points on \(\MAC\), and denote by \(\langle \varPhi(q_0), \varPhi(q_1)\rangle_{L^2}\) the \(L^2\)-Hermitian product of their images under the Madelung transform.\\
	For every \(k > 0\), let \(\mathcal{G}_k(q_0,q_1)\) denote the set of magnetic geodesics with energy \(k\) connecting \(q_0\) and \(q_1\). The following cases hold:
	\begin{enumerate}
		\item\label{it:1 mane M2HS} If \(k > \tfrac{1}{8}\), then \(\mathcal{G}_k(q_0,q_1) \neq \varnothing\).
		\item\label{it:2 mane M2HS}  If \(k = \tfrac{1}{8}\), then \(\mathcal{G}_k(q_0,q_1) \neq \varnothing\) if and only if \(\langle \varPhi(q_0), \varPhi(q_1)\rangle_{L^2} \neq 0\).
		\item\label{it:3 mane M2HS}  If \(k < \tfrac{1}{8}\), the following subcases apply:
		\begin{enumerate}
			\item\label{it:3a mane M2HS}  If \(\left\vert \langle \varPhi(q_0), \varPhi(q_1)\rangle_{L^2} \right\vert > \sqrt{1-8k}\), then \(\mathcal{G}_k(q_0,q_1) \neq \varnothing\).
			\item\label{it:3b mane M2HS}  If \(\left\vert \langle \varPhi(q_0), \varPhi(q_1)\rangle_{L^2} \right\vert = \sqrt{1-8k}\), there exist \(a_k, b_k \in \mathbb{R}\) with \(b_k > 0\) such that \(\mathcal{G}_k(q_0,q_1) \neq \varnothing\) if and only if 
			\[
			\langle \varPhi(q_0), \varPhi(q_1)\rangle_{L^2} = e^{\mathrm{i}(a_k + mb_k)}\sqrt{1-8k},
			\]
			for some \(m \in \mathbb{Z}\).
			\item\label{it:3c mane M2HS}  If \(\left\vert \langle \varPhi(q_0), \varPhi(q_1)\rangle_{L^2} \right\vert < \sqrt{1-8k}\), then \(\mathcal{G}_k(q_0,q_1) = \varnothing\).
		\end{enumerate}
	\end{enumerate}
\end{thm}
Before proving \Cref{t: Manes critical value for (M2HS)}, we outline the essential ingredients of the argument, following the strategy of \Cref{t:mane for SiL}. The technically involved analysis is carried out in \Cref{ss: app Weak magnetic geodesics are mapped to magnetic geodesics}.
\begin{proof}
The computation of Mañé's critical value for the system follows the same reasoning as in \Cref{L: Mane crit value for SIL}; the details are left to the reader.\\
Part~(\ref{it:1 mane M2HS}) follows directly from the fact that the inverse of the Madelung transform $\varPhi$ is surjective from $\U$ onto $\MAC$ and from \Cref{t:mane for SiL}.\\
To prove Parts~(\ref{it:2 mane M2HS}) and~(\ref{it:3 mane M2HS}), we first establish the following:
\begin{enumerate}[label=(\Roman*)]
	\item\label{it: 1 proof mane M2HS}
	If $(\varphi, \tau)$ is a weak magnetic geodesic of $(\MAC, \GH, \vardalp)$, then
	\[
	\gamma := \sqrt{\varphi_x}\, e^{\i \tau/2}
	\]
	is a magnetic geodesic of $(\SiL, \GL, \mathrm{d}\alpha)$. In particular, $\gamma$ is $C^2$ in time as a map from $\RR$ to $\SiL$.

	\item\label{it: 2 proof mane M2HS}
	The quantity $\GL_{\gamma}(\dot{\gamma}, \dot{\gamma})$ is conserved along the flow.
\end{enumerate}
Statement~\ref{it: 1 proof mane M2HS} is proven in \Cref{Prop:app Weak magnetic geodesics in MAC are mapped to magnetic geodesics in Sil}. The part~(\ref{it:2 mane M2HS}) now follows directly from \ref{it: 1 proof mane M2HS} and \ref{it: 2 proof mane M2HS} in combination with Part~(\ref{it: 2 Mane for SiL}) of \Cref{t:mane for SiL}.\\
Finally, Part~(\ref{it:3 mane M2HS}) follows from \ref{it: 1 proof mane M2HS} and \ref{it: 2 proof mane M2HS} together with Part~(\ref{it: 3 Mane for SiL}) of \Cref{t:mane for SiL}. This completes the proof.
\end{proof}
\appendix
\section{Weak magnetic geodesics are mapped to magnetic geodesics of $\SiL$}\label{ss: app Weak magnetic geodesics are mapped to magnetic geodesics}
\begin{prop}\label{Prop:app Weak magnetic geodesics in  MAC   are mapped to magnetic geodesics in Sil}
       Let $(\varphi,\tau)$ be a unit speed weak magnetic geodesic of $(\MAC,\GH, \varPhi^*\d\alpha)$ of strenght $s$ with initial values $(\varphi(0),\tau(0))=(\id,0)$ and $(\varphi_t(0),\tau_t(0))=(\Tilde{u},\Tilde{\rho})$.  Then map \[\gamma:[0,\infty)\longrightarrow \SiL: t\mapsto \sqrt{\varphi_x}e^{\frac{\i\tau}{2}}\] is a unit speed magnetic geodesic of $(\SiL, \GL,\d\alpha)$  of strength $s$ with initial values $\gamma(0)=\mathds{1}$ and $\dot{\gc}(0)=\frac{1}{2}\left(\tu+\i\tr\right)$. 
    \end{prop} 
     The proof consists of two steps. We first prove that $\eta$ solves the magnetic geodesic equation of $(\SiL, \GL, \mathrm{d}\alpha)$ almost everywhere, and then use a bootstrapping argument to upgrade this to a classical solution. 

For the first step, the following lemma provides a useful characterization of absolutely continuous functions and is crucial in the first step of the proof of \Cref{Prop:app Weak magnetic geodesics in MAC are mapped to magnetic geodesics in Sil}.

\begin{lem}[{\cite[Chapter 3, Lemma 1.1]{t84}}]\label{Technical Lemma 4}
Let $X$ be a Hilbert space, and let $u,g\in L^1\left([0,T),X\right)$ for some $T>0$. Then the following two conditions are equivalent:
\begin{enumerate}
    \item\label{it:1 techn lemma app}
    The function $u$ is almost everywhere equal to a primitive of $g$, that is, there exists $F\in X$ such that
    \[
    u(t)= F+\int_{0}^t g(s)\,\mathrm{d}s 
    \quad\text{for almost every } t\in[0,T).
    \]

    \item\label{it:2 techn lemma app}
    For each test function $\eta\in C_c^{\infty}((0,T), X)$ it holds that
    \[
    \int_{0}^T u(t)\eta_t(t)\,\mathrm{d}t
    = -\int_0^T g(t)\eta(t)\,\mathrm{d}t.
    \]
\end{enumerate}

If either \eqref{it:1 techn lemma app} or \eqref{it:2 techn lemma app} holds, then $u$ is almost everywhere equal to an absolutely continuous function $v:[0,T)\longrightarrow X$.
\end{lem}

\subsection{Step 1: $\eta$ is almost everywhere a magnetic geodesic of $\SiL$}
In view of \Cref{Technical Lemma 4}, it suffices to prove the following:
\begin{lem}\label{Image of magnetic geodesic in MAC is almost a magnetic geodesic in Sil}
Let $(\varphi,\tau)$ be a unit speed weak magnetic geodesic of $(\MAC,\GH, \varPhi^*\mathrm{d}\alpha)$ of strength $s$ with initial values $(\varphi(0),\tau(0))=(\id,0)$ and $(\varphi_t(0),\tau_t(0))=(\Tilde{u},\Tilde{\rho})$. Denote by $(u,\rho)$ the weak solution of \eqref{(M2HS)} in the sense of \cref{d: weak solutions}, defined through $(u\circ\varphi,\rho\circ \varphi)= (\varphi_t,\tau_t)$. Then the map
\[
\eta:[0,\infty)\longrightarrow \SiL,\qquad 
t\mapsto \int_0^t\frac{\sqrt{\varphi_x}}{2}\left(u_x\circ \varphi+\i \rho\circ \varphi\right)e^{\i\frac{\tau}{2}}\,\mathrm{d}s
\]
is absolutely continuous and solves the magnetic geodesic equation of $(\SiL,\GL,\mathrm{d}\alpha)$, that is, \eqref{e:magnetic geodesic equation in SIL}, for almost every $t>0$.
\end{lem}

    \begin{proof}

    We begin by checking that the curve $t\mapsto \eta(t)$ is a well-defined curve in $\SiL$. Since $t\mapsto(\varphi(t,\cdot),\tau(t,\cdot))$ is a curve in $\MAC$, the derivative $\varphi_x(t,\cdot)$ exists for almost every $x\in \SS^1$ and satisfies $\varphi_x(t,\cdot)\ge 0$ for all $t$. Furthermore, since $(u,\rho)$ is, by \Cref{t: global weak solutions }, a weak solution of \eqref{(M2HS)}, it follows from \ref{it:weak_sol:Linfty} in \Cref{d: weak solutions} that the maps $t\mapsto u_x(t,\cdot)$ and $t\mapsto \rho(t,\cdot)$ belong to
\[
L^{\infty}\left([0,\infty), L^2\left(\SS^1,\RR \right) \right).
\]
Hence $\eta(t)$ is well defined. Moreover, using the transformation rule for absolutely continuous functions and Fubini's theorem, we obtain for almost every $t$ that
\begin{align*}
\int_{\SS^1}\lvert \eta(t,x)\rvert^2\,\mathrm{d}x
&= \int_{\SS^1}\left\lvert\int_0^t\frac{\sqrt{\varphi_x}}{2}\left(u_x\circ \varphi+\i \rho\circ \varphi\right)e^{\i\frac{\tau}{2}}\,\mathrm{d}s\right\rvert^2 \mathrm{d}x\\
&=\GH_{(\id, 0)}\bigl((u,\rho), (u,\rho)\bigr)=1.
\end{align*}
Therefore, we have shown that $t\mapsto \eta(t)$ defines (almost everywhere) a curve in $\SiL$. Thus, by \Cref{Technical Lemma 4}, the map $t\mapsto \eta(t)$ is absolutely continuous with derivative given by
\[
\frac{\sqrt{\varphi_x}}{2}\left(u_x\circ \varphi+\i \rho\circ \varphi\right)e^{\i\frac{\tau}{2}}
\]
for almost every $t>0$.

   We now prove that the curve $t\mapsto \eta(t)$ is, for almost every $t>0$, a solution of the magnetic geodesic equation of $(\SiL,\GL,\mathrm{d}\alpha)$, that is, $\eta$ is for a.e. $t>0$ a solution of~\eqref{e:magnetic geodesic equation in SIL}. By \ref{it:2 techn lemma app} in \Cref{Technical Lemma 4}, it suffices to prove the following claim.

\begin{claim}\label{Claim 1 for image of magnt geod in G in Sil}
The curve $t\mapsto \eta(t)$ satisfies
\[
\int_{\SS^1}\left[\int_{0}^{\infty}\dot{\eta}(t,x)\,\theta_t(t)\,\mathrm{d}t
+\int_{0}^{\infty}\left(\i s\,\dot{\eta}(t,x)-(1-s\delta)\eta(t,x)\right)\theta(t)\,\mathrm{d}t\right]\Theta(x)\,\mathrm{d}x=0
\]
for all $\theta\in C_c^{\infty}\bigl((0,\infty)\bigr)$ and all $\Theta\in C^{\infty}\left(\SS^1\right)$.
\end{claim}

For the moment, assume that \Cref{Claim 1 for image of magnt geod in G in Sil} is true. Then
\[
\int_{0}^{\infty}\dot{\eta}(t,x)\,\theta_t(t)\,\mathrm{d}t
=-\int_{0}^{\infty}\left(\i s\,\dot{\eta}(t,x)-(1-s\delta)\eta(t,x)\right)\theta(t)\,\mathrm{d}t
\]
for all $\theta\in C_c^{\infty}\bigl((0,\infty)\bigr)$. Thus, by \Cref{Technical Lemma 4}, the map
\[
\eta_t:[0,\infty)\longrightarrow \SiL
\]
is absolutely continuous and the identity
\[
\eta_t(t)=\eta_t(0)+ \int_{0}^{t}\left(\i s\,\dot{\eta}(s)-(1-s\delta)\eta(s)\right)\,\mathrm{d}s
\]
holds in $L^2\left(\SS^1\right)$. In particular, for almost every $t>0$ we have
\[
\eta_{tt}(t)=\i s\,\dot{\eta}(t)-(1-s\delta)\eta(t)
\]
in $L^2\left(\SS^1\right)$. Hence $t\mapsto \eta(t)$ is, for almost every $t>0$, a solution of \eqref{e:magnetic geodesic equation in SIL}, i.e.\ it satisfies the magnetic geodesic equation on $(\SiL,\GL,\mathrm{d}\alpha)$ almost everywhere.
 \end{proof}
So it remains to prove \cref{Claim 1 for image of magnt geod in G in Sil}.

\begin{proof}[Proof of \cref{Claim 1 for image of magnt geod in G in Sil}]
In order to prove \cref{Claim 1 for image of magnt geod in G in Sil}, we first show that the map $t\mapsto \eta_t$ is absolutely continuous. By \cref{Technical Lemma 4}, it suffices to prove the following claim.

\begin{claim}\label{Helpclaim 1 for Claim for image of magnetic geod}
The curve
\[
t\mapsto \eta_t(t)=\frac{\sqrt{\varphi_x}}{2}\left(u_x+\i \rho \right)\circ \varphi \, e^{\frac{\i \tau}{2}}
\]
satisfies
\begin{align*}
&\int_{\SS^1}\int_{0}^{\infty}\frac{\sqrt{\varphi_x}}{2}\left(u_x+\i \rho \right)\circ \varphi \, e^{\frac{\i \tau}{2}}\theta_t(t)\,\mathrm{d}t\,\Theta(x)\,\mathrm{d}x\\
=&-\int_{\SS^1}\int_{0}^{\infty}\frac{\sqrt{\varphi_x}}{2}\left(\frac{1}{2}\left(u_x^2-\rho^2 \right)+u_{xt}+uu_{xx}+\i (u\rho)_x\rho_t\right)\circ \varphi \, e^{\frac{\i \tau}{2}}\theta(t)\,\mathrm{d}t\,\Theta(x)\,\mathrm{d}x
\end{align*}
for all $\theta\in C_c^{\infty}\bigl((0,\infty), L^2(\SS^1,\CC)\bigr)$ and all $\Theta\in C^{\infty}\left(\SS^1\right)$.
\end{claim}
 \begin{proof}[Proof of \cref{Helpclaim 1 for Claim for image of magnetic geod}]
Since $\sqrt{\varphi_x}$, $u_x$, and $\rho$ belong to
\[
L^{\infty}\left([0,\infty),L^2\left(\SS^1,\CC\right)\right),
\]
we may apply Fubini's theorem to obtain
\[
\int_{\SS^1}\int_{0}^{\infty}\frac{\sqrt{\varphi_x}}{2}\left(u_x+\i \rho \right)\circ \varphi\, e^{\frac{\i \tau}{2}}\theta_t(t)\,\mathrm{d}t\,\Theta(x)\,\mathrm{d}x
=\int_{0}^{\infty}\int_{\SS^1}\frac{\sqrt{\varphi_x}}{2}\left(u_x+\i \rho \right)\circ \varphi\, e^{\frac{\i \tau}{2}}\Theta(x)\,\mathrm{d}x\,\theta_t(t)\,\mathrm{d}t.
\]
Integrating by parts in $t$ yields
\begin{align}\label{equation to proof helpclaim 1}
&\int_{\SS^1}\int_{0}^{\infty}\frac{\sqrt{\varphi_x}}{2}\left(u_x+\i \rho \right)\circ \varphi\, e^{\frac{\i \tau}{2}}\theta_t(t)\,\mathrm{d}t\,\Theta(x)\,\mathrm{d}x \nonumber \\
=&-\int_{0}^{\infty}\int_{\SS^1}\frac{\mathrm{d}}{\mathrm{d}t}\left(\frac{\sqrt{\varphi_x}}{2}\left(u_x+\i \rho \right)\circ \varphi\, e^{\frac{\i \tau}{2}}\Theta(x)\right)\,\mathrm{d}x\,\theta(t)\,\mathrm{d}t.
\end{align}
Since $\Theta$ is independent of $t$, computing the time derivative in \eqref{equation to proof helpclaim 1} using the chain rule yields
\begin{align*}
&\int_{\SS^1}\int_{0}^{\infty}\frac{\sqrt{\varphi_x}}{2}\left(u_x+\i \rho \right)\circ \varphi\, e^{\frac{\i \tau}{2}}\theta_t(t)\,\mathrm{d}t\,\Theta(x)\,\mathrm{d}x \\
=&-\int_{\SS^1}\int_{0}^{\infty}\frac{\sqrt{\varphi_x}}{2}\left(\frac{1}{2}\left(u_x^2-\rho^2 \right)+u_{xt}+uu_{xx}+\i (u\rho)_x\rho_t\right)\circ \varphi\, e^{\frac{\i \tau}{2}}\theta(t)\,\mathrm{d}t\,\Theta(x)\,\mathrm{d}x.
\end{align*}
This proves \Cref{Helpclaim 1 for Claim for image of magnetic geod}.
\end{proof}
Using \Cref{Helpclaim 1 for Claim for image of magnetic geod} together with \Cref{Technical Lemma 4}, we infer that the map
\[
\eta_t:[0,\infty)\longrightarrow L^2\left(\SS^1,\CC \right)
\]
is absolutely continuous and that the identity
\[
\eta_t(t)=\eta_t(0)+\int_0^t\frac{\sqrt{\varphi_x}}{2}\left(\frac{1}{2}\left(u_x^2-\rho^2 \right)+u_{xt}+uu_{xx}+\i (u\rho)_x\rho_t\right)\circ \varphi \, e^{\frac{\i \tau}{2}}\,\mathrm{d}s
\]
holds in $L^2\left(\SS^1,\CC \right)$. 

Note that $(\varphi,\tau)$ is a weak magnetic geodesic of $(\MAC, \GH, \varPhi^*\mathrm{d}\alpha)$. Hence, for almost every $t>0$, it solves by \ref{it:weak_geod:equation} of \Cref{d: weak magnetic geodesics in MAC} the equation \eqref{e: magentic geod equation of Gm at identity}. Therefore, for almost every $t>0$ and $x\in \SS^1$, it satisfies \eqref{(M2HS)}, i.e.
\[
\begin{cases}
u_{tx}=-\frac{1}{2}u_x^2-u\,u_{xx}+\frac{1}{2}\rho^2-\bigl(s\rho+2(c^2-s\delta)\bigr),\\[2mm]
\rho_t=-(\rho u)_x+s u_x.
\end{cases}
\]

Substituting this into the previous identity yields
\begin{align}
\eta_t(t)
&=\eta_t(0)+\int_0^t\frac{\sqrt{\varphi_x}}{2}\left(\i s(u_x+\i \rho)-2(1-s\delta)\right)\circ \varphi \, e^{\frac{\i \tau}{2}}\,\mathrm{d}\Tilde{t}\nonumber\\
&=\eta_t(0)+\i s\int_0^t\frac{\sqrt{\varphi_x}}{2}\left(u_x+\i \rho\right)\circ \varphi \, e^{\frac{\i \tau}{2}}\,\mathrm{d}\Tilde{t}
-\int_0^t\frac{\sqrt{\varphi_x}}{2}\left(2(1-s\delta)\right)\circ \varphi \, e^{\frac{\i \tau}{2}}\,\mathrm{d}\Tilde{t}
\label{Equation for eta_t}
\end{align}
in $L^2\left(\SS^1,\CC \right)$.

By an argument similar to the proof of \Cref{Helpclaim 1 for Claim for image of magnetic geod}, one can also show that the curve
\[
\eta:[0,\infty)\longrightarrow \SiL,\qquad t\mapsto \sqrt{\varphi_x}\, e^{\frac{\i \tau}{2}}
\]
satisfies
\begin{equation}\label{Help equation for continuity of Varphi}
\int_{\SS^1}\int_{0}^{\infty}\frac{\sqrt{\varphi_x}}{2}\left(u_x+\i \rho \right)\circ \varphi\, e^{\frac{\i \tau}{2}}\theta(t)\,\mathrm{d}t\,\Theta(x)\,\mathrm{d}x
=-\int_{\SS^1}\int_{0}^{\infty}\sqrt{\varphi_x}\, e^{\frac{\i \tau}{2}}\theta_t(t)\,\mathrm{d}t\,\Theta(x)\,\mathrm{d}x
\end{equation}
for all $\theta\in C_c^{\infty}\bigl((0,\infty)\bigr)$ and all $\Theta\in C^{\infty}\left(\SS^1\right)$. 

Therefore, by \eqref{Help equation for continuity of Varphi} and \Cref{Technical Lemma 4}, the identity
\begin{equation}\label{equation for eta}
\eta(t)=\eta(0)+\int_0^t\frac{\sqrt{\varphi_x}}{2}\left(u_x+\i \rho \right)\circ \varphi\, e^{\frac{\i \tau}{2}}\,\mathrm{d}\Tilde{t}
\end{equation}
holds in $L^2\left(\SS^1,\CC \right)$.

Combining \eqref{equation for eta} and \eqref{Equation for eta_t}, we immediately obtain
\begin{equation}\label{eq: ode for eta holds almost everywhere}
    \ddot{\eta}=\i s\dot{\eta}-(1-s\delta)\eta
\end{equation}
for almost every $t>0$ in $L^2\left(\SS^1,\CC \right)$. This finishes the proof of \Cref{Claim 1 for image of magnt geod in G in Sil}.

\end{proof}
\subsection{Step 2: Bootstrapping}
Since $\eta$ is by \eqref{equation for eta} and \Cref{Technical Lemma 4} of Sobolev class $H^1$ in time and satisfies \eqref{eq: ode for eta holds almost everywhere} almost everywhere, we conclude that $\eta$ is of Sobolev class $H^2$ in time.\\
Thus, by the Sobolev embedding theorem combined with an analogous bootstrapping argument, we obtain that $\eta \in C^2(\RR, \SiL)$ in time. In particular, $\eta$ is a magnetic geodesic of $(\SiL, \GL, \mathrm{d}\alpha)$. This finishes the proof of \Cref{Prop:app Weak magnetic geodesics in MAC are mapped to magnetic geodesics in Sil}. \qed
\bibliographystyle{abbrv}
\bibliography{ref}

@article{l07.1,
    AUTHOR = {Lenells, Jonatan},
     TITLE = {The {H}unter-{S}axton equation describes the geodesic flow on a sphere},
   JOURNAL = {J. Geom. Phys.},
  FJOURNAL = {Journal of Geometry and Physics},
    VOLUME = {57},
      YEAR = {2007},
    NUMBER = {10},
     PAGES = {2049--2064},
      ISSN = {0393-0440,1879-1662},
   MRCLASS = {37K65 (37K25 53D25 58B20 58D05)},
  MRNUMBER = {2348278},
MRREVIEWER = {Fran\c cois\ Gay-Balmaz},
       DOI = {10.1016/j.geomphys.2007.05.003},
       URL = {https://doi.org/10.1016/j.geomphys.2007.05.003},
}

@article {AbbMacMazzPat17,
    AUTHOR = {Abbondandolo, Alberto and Macarini, Leonardo and Mazzucchelli,
              Marco and Paternain, Gabriel P.},
     TITLE = {Infinitely many periodic orbits of exact magnetic flows on
              surfaces for almost every subcritical energy level},
   JOURNAL = {J. Eur. Math. Soc. (JEMS)},
  FJOURNAL = {Journal of the European Mathematical Society (JEMS)},
    VOLUME = {19},
      YEAR = {2017},
    NUMBER = {2},
     PAGES = {551--579},
      ISSN = {1435-9855,1435-9863},
   MRCLASS = {37J45 (58E05)},
  MRNUMBER = {3605025},
MRREVIEWER = {Maria\ Letizia\ Bertotti},
       DOI = {10.4171/JEMS/674},
       URL = {https://doi.org/10.4171/JEMS/674},
}

@article{Bauer_2025,
   title={Regularity and completeness of half--{L}ie groups},
   ISSN={1435-9863},
   url={http://dx.doi.org/10.4171/JEMS/1587},
   DOI={10.4171/jems/1587},
   journal={Journal of the European Mathematical Society},
   publisher={European Mathematical Society - EMS - Publishing House GmbH},
   author={Bauer, Martin and Harms, Philipp and Michor, Peter W.},
   year={2025},
   month=jan }

@article{l07.2,
    AUTHOR = {Lenells, Jonatan},
     TITLE = {Weak geodesic flow and global solutions of the
              {H}unter-{S}axton equation},
   JOURNAL = {Discrete Contin. Dyn. Syst.},
  FJOURNAL = {Discrete and Continuous Dynamical Systems. Series A},
    VOLUME = {18},
      YEAR = {2007},
    NUMBER = {4},
     PAGES = {643--656},
      ISSN = {1078-0947,1553-5231},
   MRCLASS = {37K65 (35D05 35Q53 58B20 58D05)},
  MRNUMBER = {2318260},
MRREVIEWER = {Boris\ A.\ Khesin},
       DOI = {10.3934/dcds.2007.18.643},
       URL = {https://doi.org/10.3934/dcds.2007.18.643},
}

@article {Merry2010,
    AUTHOR = {Merry, Will J.},
     TITLE = {Closed orbits of a charge in a weakly exact magnetic field},
   JOURNAL = {Pacific J. Math.},
  FJOURNAL = {Pacific Journal of Mathematics},
    VOLUME = {247},
      YEAR = {2010},
    NUMBER = {1},
     PAGES = {189--212},
      ISSN = {0030-8730,1945-5844},
   MRCLASS = {37J50 (37J45 53C80)},
  MRNUMBER = {2718211},
MRREVIEWER = {Jos\'e\ Ant\^onio G. Miranda},
       DOI = {10.2140/pjm.2010.247.189},
       URL = {https://doi.org/10.2140/pjm.2010.247.189},
}

@article {AssBenLust16,
    AUTHOR = {Asselle, Luca and Benedetti, Gabriele},
     TITLE = {The {L}usternik-{F}et theorem for autonomous {T}onelli
              {H}amiltonian systems on twisted cotangent bundles},
   JOURNAL = {J. Topol. Anal.},
  FJOURNAL = {Journal of Topology and Analysis},
    VOLUME = {8},
      YEAR = {2016},
    NUMBER = {3},
     PAGES = {545--570},
      ISSN = {1793-5253,1793-7167},
   MRCLASS = {37J45 (58E05)},
  MRNUMBER = {3509572},
MRREVIEWER = {Thomas\ Bartsch},
       DOI = {10.1142/S1793525316500205},
       URL = {https://doi.org/10.1142/S1793525316500205},
}

@book{KrieglMichor1997,
    AUTHOR = {Kriegl, Andreas and Michor, Peter W.},
     TITLE = {The convenient setting of global analysis},
    SERIES = {Mathematical Surveys and Monographs},
    VOLUME = {53},
 PUBLISHER = {American Mathematical Society, Providence, RI},
      YEAR = {1997},
     PAGES = {x+618},
      ISBN = {0-8218-0780-3},
   MRCLASS = {58Bxx (46-02 46Gxx 46M20 58C15)},
  MRNUMBER = {1471480},
MRREVIEWER = {Olga\ Gil-Medrano},
       DOI = {10.1090/surv/053},
       URL = {https://doi.org/10.1090/surv/053},
}

@article{Vi08,
    AUTHOR = {Vizman, Cornelia},
     TITLE = {Geodesic equations on diffeomorphism groups},
   JOURNAL = {SIGMA Symmetry Integrability Geom. Methods Appl.},
  FJOURNAL = {SIGMA. Symmetry, Integrability and Geometry. Methods and
              Applications},
    VOLUME = {4},
      YEAR = {2008},
     PAGES = {Paper 030, 22},
      ISSN = {1815-0659},
   MRCLASS = {37K65 (35A30 35Q35 58D05)},
  MRNUMBER = {2393297},
MRREVIEWER = {Stephen\ Carl\ Preston},
       DOI = {10.3842/SIGMA.2008.030},
       URL = {https://doi.org/10.3842/SIGMA.2008.030},
}

@article{Maier25magneticEulerArnold,
  title = {{On geometric hydrodynamics and infinite dimensional magnetic systems}},
  author = {Levin Maier},
  journal = {arXiv preprint arXiv:2506.00544},
  year = {2025},
  eprint = {2506.00544},
  archivePrefix = {arXiv},
  primaryClass = {math.SG},
  url = {https://arxiv.org/abs/2506.00544}
}

@book{t84,
  author = {R. Temam},
  title = {Navier-Stokes Equations},
  publisher = {North-Holland Publishing Co.},
  address = {Amsterdam},
  year = {1984}
}

@article{wu11,
    AUTHOR = {Wunsch, Marcus},
     TITLE = {Weak geodesic flow on a semidirect product and global
              solutions to the periodic {H}unter-{S}axton system},
   JOURNAL = {Nonlinear Anal.},
  FJOURNAL = {Nonlinear Analysis. Theory, Methods \& Applications. An
              International Multidisciplinary Journal},
    VOLUME = {74},
      YEAR = {2011},
    NUMBER = {15},
     PAGES = {4951--4960},
      ISSN = {0362-546X,1873-5215},
   MRCLASS = {37K65 (35B10 35D30 35Q53 58B20)},
  MRNUMBER = {2810678},
MRREVIEWER = {Daniel\ Belti\c t\u a},
       DOI = {10.1016/j.na.2011.04.041},
       URL = {https://doi.org/10.1016/j.na.2011.04.041},
}

@article{AM15,
    AUTHOR = {Abbondandolo, Alberto and Majer, Pietro},
     TITLE = {A non-squeezing theorem for convex symplectic images of the
              {H}ilbert ball},
   JOURNAL = {Calc. Var. Partial Differential Equations},
  FJOURNAL = {Calculus of Variations and Partial Differential Equations},
    VOLUME = {54},
      YEAR = {2015},
    NUMBER = {2},
     PAGES = {1469--1506},
      ISSN = {0944-2669,1432-0835},
   MRCLASS = {53D22 (35Q55 37K05 70H05)},
  MRNUMBER = {3396420},
MRREVIEWER = {Daniel\ Belti\c t\u a},
       DOI = {10.1007/s00526-015-0832-3},
       URL = {https://doi.org/10.1007/s00526-015-0832-3},
}

@article{HopfRinowHalfLiegroups,
  title   = {{The Hopf--Rinow Theorem and Ma\~n\'e's Critical Value for Magnetic Geodesics on Half Lie-Groups}},
  author  = {Levin Maier and Francesco Ruscelli},
  journal = {arXiv preprint arXiv:2510.19323},
  year    = {2025},
  eprint  = {2510.19323},
  archivePrefix = {arXiv},
  primaryClass  = {math.SG},
  url     = {https://arxiv.org/abs/2510.19323}
}

@article{MaierRuscelliTonelli,
  title   = {{On Ma\~n\'e's Critical Value for Tonelli Lagrangians on Half Lie-Groups}},
  author  = {Levin Maier and Francesco Ruscelli},
  journal = {arXiv preprint arXiv:2511.13428},
  year    = {2025},
  eprint  = {2511.13428},
  archivePrefix = {arXiv},
  primaryClass  = {math.DG},
  url     = {https://arxiv.org/abs/2511.13428}
}

@article{Maier26Globalweak,
  title   = {{On Global Weak Solutions for the Magnetic Two-Component Hunter-Saxton System}},
  author  = {Levin Maier},
  journal = {arXiv preprint arXiv:2601.22088},
  year    = {2026},
  eprint  = {2601.22088},
  archivePrefix = {arXiv},
  primaryClass  = {math.AP},
  url     = {https://arxiv.org/abs/2601.22088}
}

@article{ABM23,
         AUTHOR = {Albers, P. and Benedetti, G. and Maier, L.},
     TITLE = {The {H}opf-{R}inow theorem and the {M}a\~n\'e{} critical value
              for magnetic geodesics on odd-dimensional spheres},
   JOURNAL = {J. Geom. Phys.},
  FJOURNAL = {Journal of Geometry and Physics},
    VOLUME = {214},
      YEAR = {2025},
     PAGES = {Paper No. 105521, 21},
      ISSN = {0393-0440,1879-1662},
   MRCLASS = {53D25 (37J55)},
  MRNUMBER = {4905477},
       DOI = {10.1016/j.geomphys.2025.105521},
       URL = {https://doi.org/10.1016/j.geomphys.2025.105521},
}

@article{ar61,
    AUTHOR = {Arnold, V. I.},
     TITLE = {Some remarks on flows of line elements and frames},
   JOURNAL = {Dokl. Akad. Nauk SSSR},
  FJOURNAL = {Doklady Akademii Nauk SSSR},
    VOLUME = {138},
      YEAR = {1961},
     PAGES = {255--257},
      ISSN = {0002-3264},
   MRCLASS = {57.48 (34.65)},
  MRNUMBER = {158330},
MRREVIEWER = {Y.\ N.\ Dowker},
}

@article{ar66,
    AUTHOR = {Arnold, V. I.},
     TITLE = {Sur la g\'eom\'etrie diff\'erentielle des groupes de {L}ie de
              dimension infinie et ses applications \`a{} l'hydrodynamique
              des fluides parfaits},
   JOURNAL = {Ann. Inst. Fourier (Grenoble)},
  FJOURNAL = {Universit\'e{} de Grenoble. Annales de l'Institut Fourier},
    VOLUME = {16},
      YEAR = {1966},
     PAGES = {319--361},
      ISSN = {0373-0956,1777-5310},
   MRCLASS = {57.50 (57.55)},
  MRNUMBER = {202082},
MRREVIEWER = {R.\ Hermann},
       URL = {http://www.numdam.org/item?id=AIF_1966__16_1_319_0},
}

@book{AK98,
    AUTHOR = {Arnold, Vladimir I. and Khesin, Boris A.},
     TITLE = {Topological methods in hydrodynamics},
    SERIES = {Applied Mathematical Sciences},
    VOLUME = {125},
 PUBLISHER = {Springer-Verlag, New York},
      YEAR = {1998},
     PAGES = {xvi+374},
      ISBN = {0-387-94947-X},
   MRCLASS = {58-02 (35Q30 58B25 58D05 76-02 76M30)},
  MRNUMBER = {1612569},
MRREVIEWER = {Nikolai\ K.\ Smolentsev},
}

@article{EM70,
    AUTHOR = {Ebin, David G. and Marsden, Jerrold},
     TITLE = {Groups of diffeomorphisms and the motion of an incompressible
              fluid},
   JOURNAL = {Ann. of Math. (2)},
  FJOURNAL = {Annals of Mathematics. Second Series},
    VOLUME = {92},
      YEAR = {1970},
     PAGES = {102--163},
      ISSN = {0003-486X},
   MRCLASS = {57.47 (76.00)},
  MRNUMBER = {271984},
       DOI = {10.2307/1970699},
       URL = {https://doi.org/10.2307/1970699},
}

@article{km02,
    AUTHOR = {Khesin, Boris and Misiolek, Gerard},
     TITLE = {Euler equations on homogeneous spaces and {V}irasoro orbits},
   JOURNAL = {Adv. Math.},
  FJOURNAL = {Advances in Mathematics},
    VOLUME = {176},
      YEAR = {2003},
    NUMBER = {1},
     PAGES = {116--144},
      ISSN = {0001-8708,1090-2082},
   MRCLASS = {37K65 (37K05 58B25 58D05)},
  MRNUMBER = {1978343},
MRREVIEWER = {Laurent\ Guieu},
       DOI = {10.1016/S0001-8708(02)00063-4},
       URL = {https://doi.org/10.1016/S0001-8708(02)00063-4},
}

@article{Khesin-Mis-Mod-inf-Newton,
    AUTHOR = {Khesin, Boris and Misiolek, Gerard and Modin, Klas},
     TITLE = {Geometric hydrodynamics and infinite-dimensional {N}ewton's
              equations},
   JOURNAL = {Bull. Amer. Math. Soc. (N.S.)},
  FJOURNAL = {American Mathematical Society. Bulletin. New Series},
    VOLUME = {58},
      YEAR = {2021},
    NUMBER = {3},
     PAGES = {377--442},
      ISSN = {0273-0979,1088-9485},
   MRCLASS = {35Q35 (58B20)},
  MRNUMBER = {4273106},
       DOI = {10.1090/bull/1728},
       URL = {https://doi.org/10.1090/bull/1728},
}

@article{Lennels13, 
    AUTHOR = {Lenells, Jonatan},
     TITLE = {Spheres, {K}\"ahler geometry and the {H}unter-{S}axton system},
   JOURNAL = {Proc. R. Soc. Lond. Ser. A Math. Phys. Eng. Sci.},
  FJOURNAL = {Proceedings of The Royal Society of London. Series A.
              Mathematical, Physical and Engineering Sciences},
    VOLUME = {469},
      YEAR = {2013},
    NUMBER = {2154},
     PAGES = {20120726, 19},
      ISSN = {1364-5021,1471-2946},
   MRCLASS = {53D25 (53C55)},
  MRNUMBER = {3061310},
MRREVIEWER = {Martin\ Kohlmann},
       DOI = {10.1098/rspa.2012.0726},
       URL = {https://doi.org/10.1098/rspa.2012.0726},
}

@article{KMM19,
    AUTHOR = {Khesin, Boris and Misiolek, Gerard and Modin, Klas},
     TITLE = {Geometry of the {M}adelung transform},
   JOURNAL = {Arch. Ration. Mech. Anal.},
  FJOURNAL = {Archive for Rational Mechanics and Analysis},
    VOLUME = {234},
      YEAR = {2019},
    NUMBER = {2},
     PAGES = {549--573},
      ISSN = {0003-9527,1432-0673},
   MRCLASS = {58D05 (35Q31 35Q35 76N17 81Q35)},
  MRNUMBER = {3995046},
MRREVIEWER = {Chi-Kun\ Lin},
       DOI = {10.1007/s00205-019-01397-2},
       URL = {https://doi.org/10.1007/s00205-019-01397-2},
}

@article{BM24,
    AUTHOR = {Bimmermann, Johanna and Maier, Levin},
     TITLE = {Magnetic billiards and the {H}ofer-{Z}ehnder capacity of disk
              tangent bundles of lens spaces},
   JOURNAL = {Math. Ann.},
  FJOURNAL = {Mathematische Annalen},
    VOLUME = {392},
      YEAR = {2025},
    NUMBER = {4},
     PAGES = {5209--5233},
      ISSN = {0025-5831,1432-1807},
   MRCLASS = {53D05},
  MRNUMBER = {4958501},
       DOI = {10.1007/s00208-025-03212-8},
       URL = {https://doi.org/10.1007/s00208-025-03212-8},
}

@article{ABE24,
         AUTHOR = {Abbondandolo, Alberto and Benedetti, Gabriele and Edtmair,
              Oliver},
     TITLE = {Symplectic capacities of domains close to the ball and
              {B}anach-{M}azur geodesics in the space of contact forms},
   JOURNAL = {Duke Math. J.},
  FJOURNAL = {Duke Mathematical Journal},
    VOLUME = {174},
      YEAR = {2025},
    NUMBER = {8},
     PAGES = {1567--1646},
      ISSN = {0012-7094,1547-7398},
   MRCLASS = {53D35 (37Jxx 53C22 53D05)},
  MRNUMBER = {4916111},
MRREVIEWER = {St\'ephane\ Tchuiaga},
       DOI = {10.1215/00127094-2024-0066},
       URL = {https://doi.org/10.1215/00127094-2024-0066},
}

@article{Co06,
	author = {Contreras, Gonzalo},
	doi = {10.1007/s00526-005-0368-z},
	fjournal = {Calculus of Variations and Partial Differential Equations},
	issn = {0944-2669,1432-0835},
	journal = {Calc. Var. Partial Differential Equations},
	mrclass = {37J50 (37J45 58E05 70H03)},
	mrnumber = {2260806},
	mrreviewer = {Karl\ Friedrich\ Siburg},
	number = {3},
	pages = {321--395},
	title = {The {P}alais-{S}male condition on contact type energy levels for convex {L}agrangian systems},
	url = {https://doi.org/10.1007/s00526-005-0368-z},
	volume = {27},
	year = {2006}}

@book{La99,
    AUTHOR = {Lang, Serge},
     TITLE = {Fundamentals of differential geometry},
    SERIES = {Graduate Texts in Mathematics},
    VOLUME = {191},
 PUBLISHER = {Springer-Verlag, New York},
      YEAR = {1999},
     PAGES = {xviii+535},
      ISBN = {0-387-98593-X},
   MRCLASS = {53-01 (58-01)},
  MRNUMBER = {1666820},
MRREVIEWER = {Man\ Chun\ Leung},
       DOI = {10.1007/978-1-4612-0541-8},
       URL = {https://doi.org/10.1007/978-1-4612-0541-8},
}

@article{CIPP,
	author = {Contreras, G. and Iturriaga, R. and Paternain, G. P. and Paternain, M.},
	doi = {10.1007/s000390050074},
	fjournal = {Geometric and Functional Analysis},
	issn = {1016-443X,1420-8970},
	journal = {Geom. Funct. Anal.},
	mrclass = {58F05 (58F17 58F27)},
	mrnumber = {1650090},
	mrreviewer = {H\'ector\ S\'anchez-Morgado},
	number = {5},
	pages = {788--809},
	title = {Lagrangian graphs, minimizing measures and {M}a\~n\'e's critical values},
	url = {https://doi.org/10.1007/s000390050074},
	volume = {8},
	year = {1998}}

@incollection{Man,
	author = {Ma\~n\'e, Ricardo},
	booktitle = {International {C}onference on {D}ynamical {S}ystems ({M}ontevideo, 1995)},
	isbn = {0-582-30296-X},
	mrclass = {58F05 (58F11)},
	mrnumber = {1460800},
	mrreviewer = {Maria\ Letizia\ Bertotti},
	pages = {120--131},
	publisher = {Longman, Harlow},
	series = {Pitman Res. Notes Math. Ser.},
	title = {Lagrangian flows: the dynamics of globally minimizing orbits},
	volume = {362},
	year = {1996}}

@book{Pet,
	author = {Petersen, Peter},
	doi = {10.1007/978-3-319-26654-1},
	edition = {Third},
	isbn = {978-3-319-26652-7; 978-3-319-26654-1},
	mrclass = {53-01 (53C20 53C21 53C23)},
	mrnumber = {3469435},
	pages = {xviii+499},
	publisher = {Springer, Cham},
	series = {Graduate Texts in Mathematics},
	title = {Riemannian geometry},
	url = {https://doi.org/10.1007/978-3-319-26654-1},
	volume = {171},
	year = {2016}}

@article{APB,
	author = {\'Alvarez Paiva, J. C. and Balacheff, F.},
	doi = {10.1007/s00039-014-0250-2},
	fjournal = {Geometric and Functional Analysis},
	issn = {1016-443X,1420-8970},
	journal = {Geom. Funct. Anal.},
	mrclass = {53D10 (37G05 53C23 53C60)},
	mrnumber = {3192037},
	mrreviewer = {Neil\ N.\ Katz},
	number = {2},
	pages = {648--669},
	title = {Contact geometry and isosystolic inequalities},
	url = {https://doi.org/10.1007/s00039-014-0250-2},
	volume = {24},
	year = {2014}}

@article {KhesinMisilokeModinMadelung-proc,
    AUTHOR = {Khesin, Boris and Misiolek, Gerard and Modin, Klas},
     TITLE = {Geometric hydrodynamics via {M}adelung transform},
   JOURNAL = {Proc. Natl. Acad. Sci. USA},
  FJOURNAL = {Proceedings of the National Academy of Sciences of the United
              States of America},
    VOLUME = {115},
      YEAR = {2018},
    NUMBER = {24},
     PAGES = {6165--6170},
      ISSN = {0027-8424,1091-6490},
   MRCLASS = {58D30 (58D05 60B99 76N10)},
  MRNUMBER = {3829316},
MRREVIEWER = {Leandro\ A.\ Lichtenfelz},
       DOI = {10.1073/pnas.1719346115},
       URL = {https://doi.org/10.1073/pnas.1719346115},
}

@incollection{Gin,
	AUTHOR = {Ginzburg, Viktor L.},
     TITLE = {On closed trajectories of a charge in a magnetic field. {A}n
              application of symplectic geometry},
 BOOKTITLE = {Contact and symplectic geometry ({C}ambridge, 1994)},
    SERIES = {Publ. Newton Inst.},
    VOLUME = {8},
     PAGES = {131--148},
 PUBLISHER = {Cambridge Univ. Press, Cambridge},
      YEAR = {1996},
      ISBN = {0-521-57086-7},
   MRCLASS = {58F22 (58F05)},
  MRNUMBER = {1432462},
MRREVIEWER = {Iskander\ A.\ Taimanov},}

@Article{HopfrinowfalseEkeland,
    AUTHOR = {Ekeland, Ivar},
     TITLE = {The {H}opf-{R}inow theorem in infinite dimension},
   JOURNAL = {J. Differential Geom.},
  FJOURNAL = {Journal of Differential Geometry},
    VOLUME = {13},
      YEAR = {1978},
    NUMBER = {2},
     PAGES = {287--301},
      ISSN = {0022-040X,1945-743X},
   MRCLASS = {58B20 (58E10)},
  MRNUMBER = {540948},
MRREVIEWER = {Karsten\ Grove},
       URL = {http://projecteuclid.org/euclid.jdg/1214434494},
}

@article{HopfrinowfalseAktkin,
    author = {Atkin, C. J.},
    title = "{The Hopf-Rinow Theorem is false in infinite Dimensions}",
    journal = {Bulletin of the London Mathematical Society},
    volume = {7},
    number = {3},
    pages = {261-266},
    year = {1975},
    month = {11},
    issn = {0024-6093},
    doi = {10.1112/blms/7.3.261},
    url = {https://doi.org/10.1112/blms/7.3.261},
    eprint = {https://academic.oup.com/blms/article-pdf/7/3/261/6694477/7-3-261.pdf},
}

@article{BKmag,
	author = {Benedetti, G. and Kang, J.},
	fjournal = {The Journal of Symplectic Geometry},
	issn = {1527-5256},
	journal = {J. Symplectic Geom.},
	mrclass = {53D05},
	mrnumber = {4518249},
	number = {1},
	pages = {99--134},
	title = {On a systolic inequality for closed magnetic geodesics on surfaces},
	volume = {20},
	year = {2022}}

@book{Gg08,
  title={An introduction to contact topology},
  author={Geiges, Hansj{\"o}rg},
  volume={109},
  year={2008},
  publisher={Cambridge University Press}
}
\end{document}